\documentclass[a4paper,11pt]{amsart} 
\usepackage{MiscStyle}
\usepackage{NumberArtStyle}
\usepackage{ShortcutStyle}
\usepackage{tikz-cd}
\usepackage{verbatim}

\newcommand{\Inst}{\mathrm{Inst}}

\newcommand{\Tot}{\mathrm{Tot}}

\newcommand{\sing}{\mathrm{sing}}

\begin{document}

\author{Yongqiang Liu,  Guillermo Pe\~nafort Sanchis, Matthias Zach}

\title{Cohomological connectivity of perturbations of map-germs}

\address{Yongqiang Liu: The Institute of Geometry and Physics, University of Science and Technology of China, 96 Jinzhai Road, Hefei 230026 P.R. China} 
\email{liuyq@ustc.edu.cn}
\address{Guillermo Pe\~nafort Sanchis: Mathematics Department, Universitat de València, Dr. Moliner, 50,
46100 Burjassot - Spain}
\email{guillermo.penafort@uv.es}
\address{Matthias Zach: Institut f\"ur Algebraische Geometrie, Leibniz Universit\"at Hannover, Welfengarten 1, 30167 Hannover - Germany}
\email{zach@math.uni-hannover.de}

\subjclass[2010]{58K15, 58K60, 32S30, 57R45}
\thanks{The first author is partially supported by the starting grant KY2340000123 from University of Science and Technology of China, National Natural Science Funds of China (Grant No. 12001511) and the project ``Analysis and Geometry on Bundles" of Ministry of Science and Technology of the People`s Republic of China.\newline
The second author is partially supported by the ERCEA 615655 NMST Consolidator Grant and 
by the Basque Government through the BERC 2014-2017 program, the Spanish Ministry of Economy and
Competitiveness MINECO: BCAM Severo Ochoa excellence accreditation SEV-2013-0323 and by Programa de Becas Posdoctorales en la UNAM, DGAPA, Instituto de Matem\'aticas, UNAM}

	\begin{abstract}
Let $f\colon (\mathbb C^n,S)\to (\mathbb C^p,0)$ be a finite map-germ with $n<p$ and $Y_\delta$ the image of a small perturbation $f_\delta$. We show that
the reduced cohomology of $Y_\delta$ is concentrated in a range of degrees determined by the dimension of
the instability locus of $f$. In the case $n\geq p$ we obtain an analogous result, replacing finiteness by $\mathcal K$-finiteness and $Y_\delta$ by the discriminant $\Delta(f_\delta)$.  We also study the monodromy associated to the perturbation $f_\delta$.
	\end{abstract}


\maketitle

\section{Introduction}

In this paper we establish bounds for the vanishing cohomology for images and
discriminants of map-germs with non-isolated
instabilities. As will become clear, these results are parallel to
the classical
 bounds of Kato and Matsumoto on the cohomology of the Milnor fiber of
non-isolated hypersurface singularities \cite{KatoMatsumoto75}.

It is also our intention to illustrate that perverse sheaves are a powerful tool in the study of
singular mappings. For the reader who is unfamiliar with the machinery of
perverse sheaves, \cite{D2} provides a quick introduction. Here we have
included only the information that is necessary for our applications,
and have avoided technical definitions. For those who are already
well-versed in the topic, some non-trivially perverse sheaves related to
the alternating cohomology of multiple points are introduced in Section
\ref{secMultPoints}. The perversity of these sheaves was discovered by
Houston \cite{Houston00} .

 In this paper, all homology and cohomology groups are assumed to have 
 complex coefficients, unless otherwise specified.

\subsection{The Milnor fibration} Consider a germ of a
non-constant holomorphic function $g \colon (\C^{n+1},0) \to (\C,0)$ and 
a representative of $g$ thereof, defined on some open neighborhood of 
a closed ball $B_r$ of sufficiently small radius
$r$ centered at the origin. Consider a punctured
open disc $D^*$, centered at the origin in $\C$ and of radius $\delta>0$ 
with $\delta$ sufficiently small with respect to $r$.
This gives rise to a locally trivial fibration
\[
  g \colon   B_r \cap g^{-1}(D^*) \to D^*,
\]
called the \emph{Milnor fibration}, whose generic fibre
\[
 M=M_g(0)=g^{-1}(\delta)\cap B_r.
\] 
is known as the \emph{Milnor fibre} of $g$ at $0 \in \C^{n+1}$. Milnor showed \cite{Mil68} that if $g^{-1}(0)$ has an isolated
singularity at $0$, then $M$ is homotopy equivalent to a bouquet
of $n$-dimensional spheres. In particular, the reduced cohomology
of $M$---that is the vanishing cohomology of $g$ at $0$---is concentrated in the middle
degree.  Later, Kato and Matsumoto \cite{KatoMatsumoto75} established
their results for functions with non-isolated singularities: If the
critical locus $\Sigma(g)$ has dimension $d$, then the Milnor fiber $M$
is at least $(n-d-1)$-connected.  Consequently, the reduced cohomology $\tilde
H^p(M)$ is concentrated in the range of degrees
$n-d \leq p \leq n$. This concentration of reduced cohomology is what we refer to as \emph{cohomological
connectivity}.

Besides the cohomological version of Kato's and Matsumoto's connectivity
result, we wish to also study the monodromy
transformations. Since our setting will be slightly different from that
of the Milnor fibration, we briefly review the general notion of
monodromy for a topologically locally trivial fibration with fiber 
$F$ 
\[
  E\stackrel{\pi}{\longrightarrow} D^*
\] 
over a punctured disk $D^*$. Let $\exp \colon S \to D^*$ be the 
universal cover of $D^*$ by an infinite strip $S$. Since $S$ is simply 
connected, we may choose a \textit{global} trivialization of the 
the fiber product 
\[
  E' := E \times_{D^*} S \cong F \times S
\]
and deliberately identify the fiber $F_\delta = \pi^{-1}(\delta)$ 
over a point $\delta \in D^*$ with 
any fiber of $E'$ over a point $\delta'$ mapping to $\delta \in D$.
The chosen global trivialization of $E'$ over $S$
furnishes a notion of parallel transport of the fiber $F$ over 
$D^*$ which is unique up to homotopy and independent of the choices made.
The monodromy operator is defined to 
be the homeomorphism 
\[
  h \colon F \to F
\]
obtained from parallel transport along a closed loop in $D^*$ passing once 
counterclockwise around the origin. It follows that the induced maps on 
cohomology 
\[
  h^i \colon H^i(F) \to H^i(F)
\]
are well defined automorphisms. 

Let us recall the classical monodromy theorem in the Milnor fibration setting, see {\cite[Theorem 3.1.20]{D1}: 
\begin{thm} Let $h^i\colon H^i(M)\to H^i(M)$ be the $i$-th monodromy operator on the Milnor fibre associated to a germ $g\colon (\C^{n+1},0)\to(\C,0)$ (not necessarily with isolated singularities). Then 
\begin{enumerate}
\item[(1)]  the eigenvalues of $h^i$ are roots of unity. 
\item[(2)] the size of the Jordan blocks of $h^i$ is at most $i+1$. 
\end{enumerate}
\end{thm}
It is worth noting that the first item is valid in the more general setting of the fibration induced by any analytic map-germ $g\colon (X,0)\to (\C,0)$, where $X$ is an arbitrary analytic space \cite{Le78}. 

\subsection{The fibration associated to an unfolding}
Rather than looking at germs of hypersurfaces, we will
be looking at singularities that arise from multigerms of mappings 
\[
  f \colon (\C^n,S) \to (\C^p,0).
\]
Here $S$ is a \emph{finite} subset of $\C^n$, mapped to $0$ by $f$. Moreover, we always assume $f$ to be $\cK$-finite (for its definition, see Definition 2.1).

Whenever $p>n$, our attention will be directed towards the image of $f$, denoted by
\[
  (Y,0)=(\im f,0) \subset (\C^p,0).
\]
For obvious reasons, the analytic space $(Y,0)$ is sometimes called a 
\emph{parametrizable singularity}. 

Whenever $n\geq p$, $\cK$-finite mappings are surjective and there is 
no interest in studying the image of $f$ in this case. The attention 
is directed towards the \textit{discriminant} instead:
the \emph{critical locus} of $f$ is the germ 
\[
  \Sigma(f)=\{x\in \C^n\mid
\mathrm df_x\text{ is not surjective}\}
\]
and the discriminant of $f$ is defined to be the image
\[
  \Delta(f)=f(\Sigma(f)).
\] 
Observe that the study of discriminants comprises the case of
parametrizable singularities in the sense that, for $p>n$, the
differential cannot be surjective and hence $\Sigma(f)=\C^n$ and
$Y=\Delta(f)$.

\medskip

Let us summarize how the classical Milnor fibration is replaced by the fibration determined by a
one-parameter family in the context of map germs: 
A \emph{ one-parameter unfolding} $F = (f_t,t)$ of $f$ is a germ
\[
  F\colon (\C^n\times \C,S\times\{0\})\to(\C^p\times \C,0),
\]
of the form $F(x,t)=(f_t(x),t)$ and such that $f_0=f$. 
In the case $n<p$, the projection from the image  $Y=\im F$ to the
parameter space gives a fibration, the generic fiber of which 
is $Y_\delta$, the image of a perturbation $f_\delta$. 
In the case $n\geq p$ we consider the fibration defined on the 
discriminant $\Delta(F)$ whose generic
fibre is $\Delta(f_\delta)$. These fibrations depend not only on the map
germ $f$, but also on the chosen unfolding $F$. Special attention is paid
to the  case where the perturbations $f_\delta$ are stable. In this case,
the generic fibre $Y_\delta$ (or $\Delta(f_\delta)$) plays a role closer
to that of the classical Milnor fibre. In the case $n<p$, the image
$Y_\delta$ of a stable perturbation is called a \emph{disentanglement}
of $f$. For the precise definition of these terms see Section 3.1.

Unfoldings play the role of deformations and are still local objects. To define perturbations---the nearby objects---we need a well chosen representative. It is customary to absorb all associated technicalities 
in the definition of a \textit{good representative}. 

\begin{definition}\label{defGoodRep}\label{defPerturbation}
  Let $V \subset\C^p$ and $T \subset \C$ be open neighborhoods of the origin 
  and 
  \[
    F\colon W \to V\times T
  \]
  a representative of a one parameter unfolding 
 defined on an open subset $W \subset \C^n\times \C$. We call 
 $F$ a \emph{good representative} if it is a $\cK$-finite map and moreover 
 satisfies the following conditions:
  \begin{enumerate}
    \item the family $\pi \colon \Delta(F)\to T$ is a locally trivial fibration 
      over $T\setminus\{0\}$,  
    \item the central fiber $\pi^{-1}(0)$ is contractible,
    \item the space $\Delta(F)$ retracts onto the central fiber.
  \end{enumerate}
  For any fixed nonzero value $\delta\in T\setminus\{0\}$ in the parameter space of
  a good representative the map 
  \[
    f_\delta\colon W_\delta \to V
  \] 
  on $W_\delta := W \cap (\C^n \times \{\delta\})$ 
  is called a \emph{perturbation} of $f$. With no risk of confusion, we also write $f$
  for the representative $f_0\colon W_0\to V$.
\end{definition}

\begin{rem}
 A good representative of a $\cK$-finite germ can be obtained from an arbitrary representative $F \colon U \times T \to V \times T$ as follows: 
The discriminant $ \Delta(F)$ in $V \times T$ is a closed 
complex analytic set and the projection $\pi \colon V \times T \to T$ 
a holomorphic function on it.
The Milnor-L\^e fibration asserts that 
for sufficiently small ball $B_r \subset V$ around the origin 
and a subsequently chosen disc 
$D_\delta \subset T$ with 
$r \gg \delta >0$ the restriction 
\[
  \pi \colon \Delta(F) \cap (B_r \times D_\delta) \to D_\delta,
\]
is a map satisfying the three properties mentioned above.
Now the good representative is furnished by choosing 
$B_r \times D_\delta$ small enough such that the restriction 
\[
  F|_{F^{-1}(B_r \times D_\delta)} \colon F^{-1}(B_r\times D_\delta) \to B_r \times D_\delta
\]
is a $\cK$-finite map, replacing $V\times T$ by $B_r \times D_\delta$, and 
finally setting $W := F^{-1}(B_r \times D_\delta)$.
\end{rem}

Parametrizable singularities and even their disentanglements are usually 
highly singular. This is due to the fact that the image of a map $f$ may be singular 
even if $f$ is stable, that is when $f$ does not admit any non-trivial unfoldings; 
see the examples below. 
Thus, the singular locus of 
$(Y,0)$ is not well suited 
to be the analogue of the critical locus of $g$ from the classical 
Kato-Matsumoto result. Instead, we will be considering the 
\emph{instability locus} 
\[
  \Inst(f)\subseteq \Delta(f)
\]
which is an analytic subset of the discriminant of $f$ (see Section \ref{secStability}). The bound 
on the vanishing cohomology of the disentanglement will be given 
in terms of the dimension $d$ of $\Inst(f)$.
We give some examples to illustrate the situation.

\begin{ex}\label{exDoublePoint}
Consider the function 
\[
  g \colon \C^3 \to \C, \quad (x,y,z) \mapsto xy
\]
and the associated hypersurface $X = g^{-1}(0)$.
Since $z$ is a dummy coordinate, the Milnor fiber $M_g(0)$ has the homotopy type of a circle.
This is consistent with Kato and Matsumoto's theorem, because the critical locus 
of $g$, and hence also the singular locus of $X$, has dimension one, see Figure \ref{figNodeTimesIntervalMilnor}. 
 \begin{figure}
\begin{center}
\includegraphics[scale=0.9]{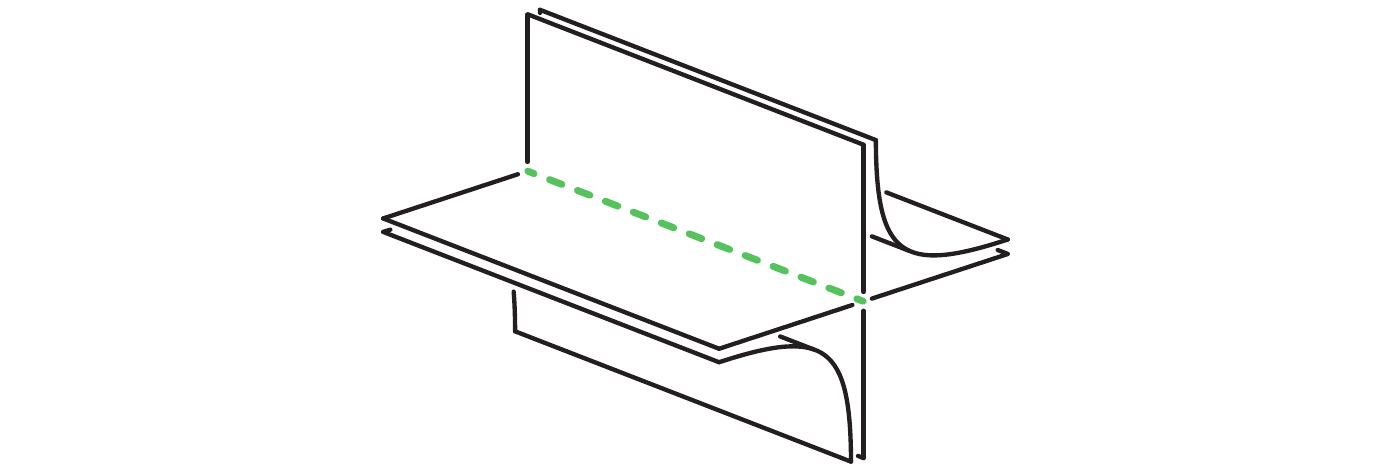}
\end{center}
\caption{The singularity $X=\{xy=0\}$ and its Milnor fibre $\{xy=\delta\}$. The dashed green line is the singular locus of $X$.}
\label{figNodeTimesIntervalMilnor}
\end{figure}

The situation changes when we think of $(X,0)$ as a parametrized singularity 
given by a bi-germ 
\[
  f \colon (\C^2,\{p,q\}) \to (\C^3,0)
\]
with the two obvious branches, see Figure \ref{figNodeTimesIntervalParametrization}. 

 \begin{figure}
\begin{center}
\includegraphics[scale=0.9]{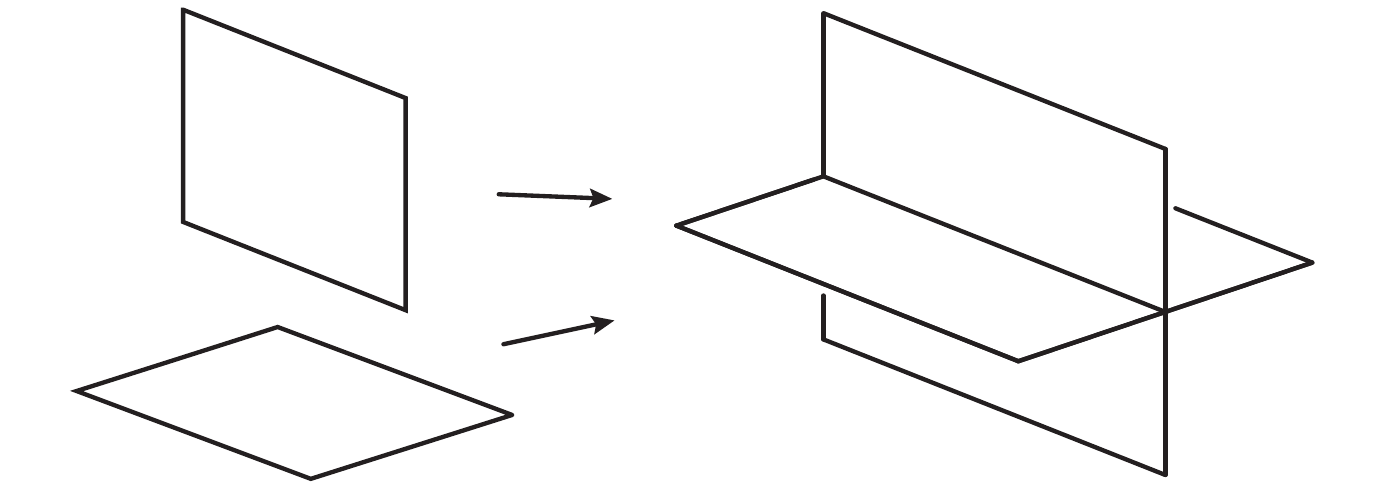}
\end{center}
\caption{The same singularity as in Figure \ref{figNodeTimesIntervalMilnor}, now regarded as the image of a transverse double point bi-germ.}
\label{figNodeTimesIntervalParametrization}
\end{figure}
This map germ is known as a ``transverse double point''.
It is stable, meaning that it cannot be perturbed by any unfolding, 
up to analytic isomorphism. Therefore, a sufficiently small representative of the image of $f$ coincides with its 
disentanglement. We see that, unlike the Milnor fiber $M_g(0)$ of $g$, 
the disentanglement is a singular space with two smooth branches crossing transversally and it has trivial reduced cohomology.
 \end{ex}

\begin{ex}\label{exFinDet}We will now consider map germs $f$ with
 \emph{isolated instability}, i.e.
   those for which $d = \dim \Inst(f) = 0$. 
   This property turns out to be equivalent to $f$ being 
   \emph{finitely determined} \cite{MatherIII}.
   The families $S_k,B_k$ and $H_k$ of Mond \cite{Mond1985On-the-clasific} are 
   examples of such finitely determined germs $(\C^2,0) \to (\C^3,0)$.
   Here we consider the germ 
   \[
     (x,y)\mapsto(x^2,y^2,x^3+y^3+xy),
   \]
   depicted in Figure \ref{figDoubleFoldTrifold}.
   
    \begin{figure}
\begin{center}
\includegraphics[scale=0.9]{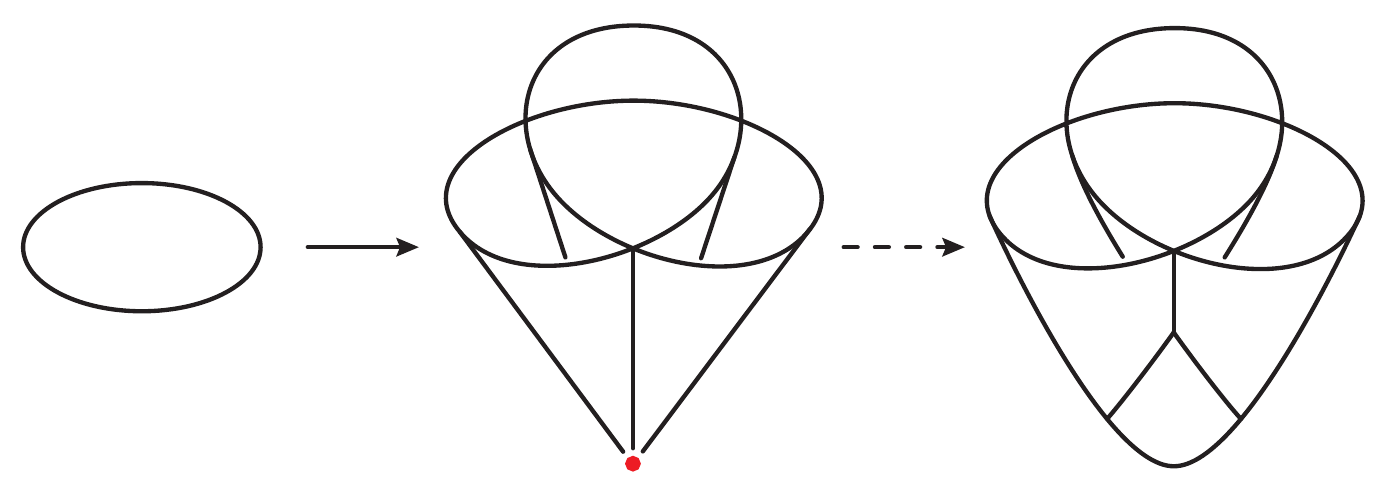}
\end{center}
\caption{A germ $f$ with isolated instability and its perturbation $f_\delta$. The red thick dot at the vertex of the cone represents the instability locus $\Inst(f)$.}
\label{figDoubleFoldTrifold}
\end{figure}

   Mond showed that the disentanglement of a finitely determined germ
   $(\C^n,0)\to(\C^{n+1},0)$ has the homotopy type of a bouquet of
   $n$-dimensional spheres \cite{Mond:1991}. Moreover, he showed that
   the number of spheres is independent of the chosen stabilization. 
Moreover, all finitely determined map-germs $f\colon(\C^n,0)\to(\C^{n+1},0)$ with
  $n\leq 14$ admit perturbations to stable mappings\footnote{
    The pairs $(n,p)$ of dimensions where every map-germ $(\C^n,S)\to(\C^p,0)$
    admits a stable perturbation are called Mather's nice dimensions. 
    These comprise all pairs $(n,n+1)$ with $n\leq 14$, cf. 
    \cite[Section 5.2.2]{MNB20}).  
  }
  $f_\delta$. 
  
  A similar connectivity result holds for
  discriminants of perturbations of finitely determined map-germs $f\colon
  (\C^n,0)\to(\C^p,0)$, when $n\geq p$ \cite{Damon:1991}.  The discriminant $\Sigma(f_\delta)$ has the homotopy type of a bouquet of spheres of dimension $p-1$. 

\end{ex}

\begin{ex}\label{exCuspEdge}
  The cuspidal edge $f\colon(\C^2,0)\to (\C^3,0)$, given by
  \[
 (x,y)\mapsto(x,y^2,y^3),
  \] 
  parametrizes the
  cartesian product of a usual cusp and the complex line, see Figure \ref{figAdjacenciesBinfty}. 
  
 It is not finitely
  determined, because it has a line of instabilities. 
 Maps with non-isolated 
  instability locus  may admit more than one disentanglement:
  The cuspidal edge can, for example, be perturbed
  into a cuspidal node 
  \[
    (x,y)\mapsto(x,y^2,y^3-\delta y)
  \] 
  which is stable
  and has the homotopy type of a circle. Another perturbation of $f$ is
   \[
    (x,y)\mapsto(x,y^2,y^3+\delta y(x^2-\delta)).
  \]
  The image of this last one has the homotopy type of a bouquet of
  two-dimensional spheres. 

\end{ex}
     \begin{figure}\label{figAdjacenciesBinfty}
\begin{center}
\includegraphics[scale=0.90]{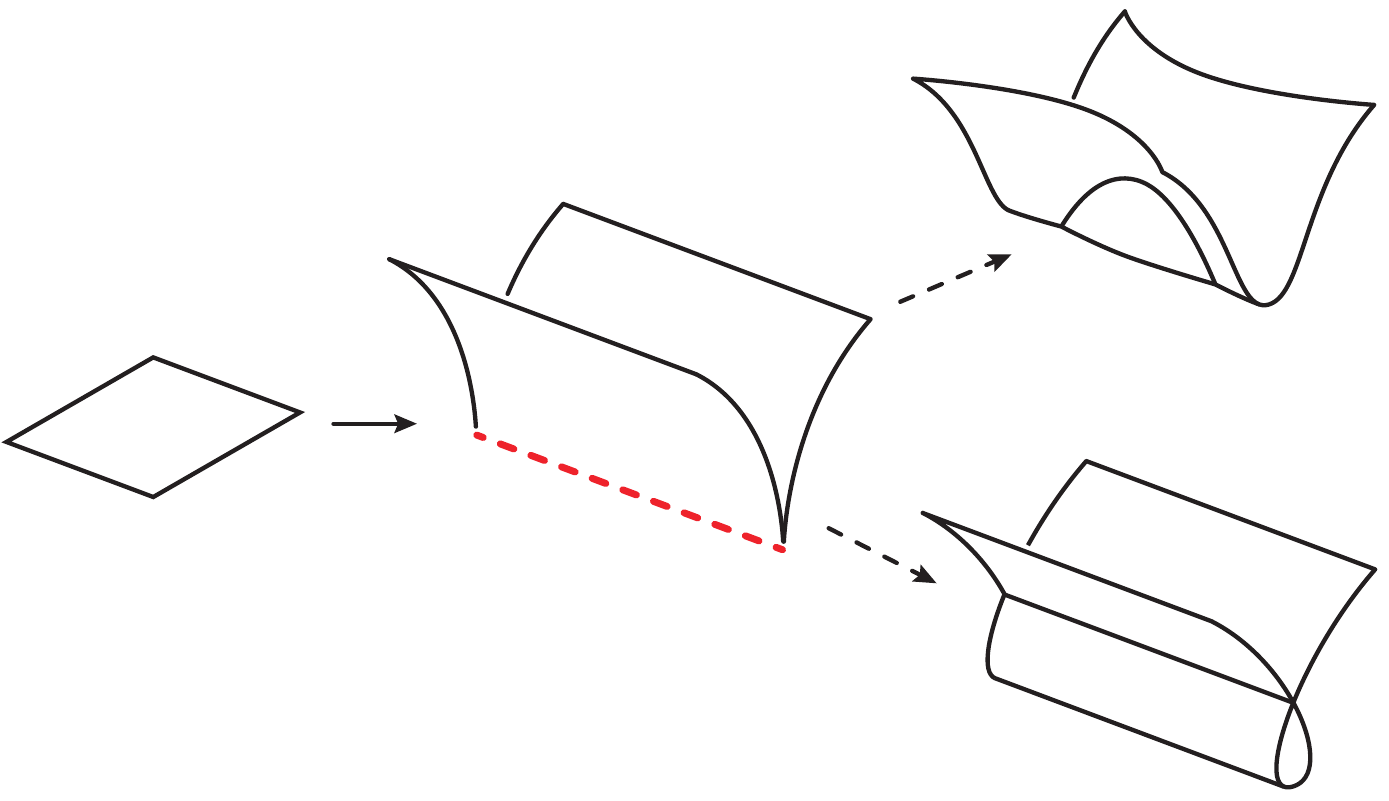}
\end{center}
\caption{The cuspidal edge and two different perturbations. The red dashed line represents the instability locus $\Inst(f)$.}
\label{figAdjacenciesBinfty}
\end{figure}

The main point of this article is to show that 
despite the fact that different unfoldings may lead to different 
disentanglements with possibly distinct homotopy types, 
there is \emph{always} a bound on 
the degrees of the nontrivial reduced cohomology groups of \emph{any} given disentanglement. 
As the three examples above suggest, this bound 
is not related to the dimension of the singular locus of the image 
(equal to one for all the germs $(\C^2,0)\to(\C^3,0)$ in our examples), 
but it is controlled by the dimension of the instability locus 
(empty, zero-dimensional and one-dimensional, respectively in the examples). 


Before stating our results, we shall introduce a construction that
connects the Milnor fibration to disentanglement fibration. Indeed,
we will show that the study of the Milnor fibration of germs of
hypersurfaces in $\C^n$ is equivalent to that of disentanglements of
bi-germs of immersions $(\C^n,\{p,q\})\to (\C^{n+1},0)$.

\begin{ex}\label{exConstructionBiGerm}
To any hypersurface $X=V(g)\subseteq (\C^n,0)$,  not necessarily with isolated singularity, we are going to associate a bi-germ of immersion $(\C^n,0)\sqcup (\C^n,0)\to\C^{n+1}$. For convenience, label the two copies of $(\C^n,0)$ as $U_1$ and $U_2$, then let $f\colon U_1\sqcup U_2\to (\C^{n+1},0)$ be the map of the form
\[
   x \mapsto 
  \begin{cases}
    (x,g(x)) & \textnormal{ if } x \in U_1 \\
    (x,0) & \textnormal{ if } x \in U_2
  \end{cases}
\]
A different choice $g'$ of a generator of the ideal $\langle g\rangle$ gives rise to a different map $f'$, but there is a change of coordinates in $(\C^{n+1},0)$ turning $f$ into $f'$. In other words, $f$ and $f'$ are $\mathcal A$-equivalent and, consequently, the study of their disentanglements is equivalent. 

Conversely, every bi-germ of immersion $(\C^n,S)\to (\C^{n+1},0)$
arises---up to $\mathcal A$-equivalence---by this construction: Given
such a bi-germ, we can take a change of coordinates turning the second
branch into $x\mapsto (x,0)$, and so that the normal vector to the first
branch at the origin has a nonzero last coordinate in $\C^{n+1}$. This
makes the first branch locally into a graph, as desired. A more direct
way to invert the process is as follows: The two branches of a bi-germ
$f$ of immersion are two germs $f^{(1)},f^{(2)}\colon (\C^n,0)\to
(\C^{n+1},0)$. Take a reduced equation $L=0$ for the image of the first
branch. Applying the above contruction to the function germ $g=L\circ
f^{(2)}\colon (\C^n,0)\to (\C,0)$ gives rise to a bi-germ
which is $\cA$-equivalent to the original immersion $f$. 
 
 Having explained the construction and its inverse, let us describe the relation between the hypersurface $X$ and the immersion $f$: First and most obvious, the intersection of the two branches is
 \[\im (f\vert_{U_1})\cap \im (f\vert_{U_2})=X\times\{0\}.\]
 
Moreover, the two branches cross transversally, except on the singularities of $X$. It is well known that the instabilities of an immersion  are located precisely at points where the branches do not intersect in general position (this is a particular case of Theorem 3.3 in \cite{MNB20}, taking into account that monogerms of immersions are stable). In particular, it follows that
\[\Inst (f)=\Sing X\times \{0\}.\]

The Milnor fibre of $X$ has the form $M=g^{-1}(\{\delta\})\cap B_r$, for suitable choices of $r \gg \delta >0$. The same discussion about singularities and stabilities  shows that a stabilization of $f$ is given by the bi-germ $F=(f_t,t)$, with \[f_t\vert_{U_1}=f\vert_{U_1}\quad\text{and}\quad f_t\vert_{U_2}=(x,t).\] A stable perturbation of $f$ is given by $f_\delta\colon B_r\sqcup B_r\to \C^{n+1}$ and the Milnor fibre  $M$ of $g$ is recovered as the intersection of the two branches of $f_\delta$. 
     \begin{figure}
\begin{center}
\includegraphics[scale=0.90]{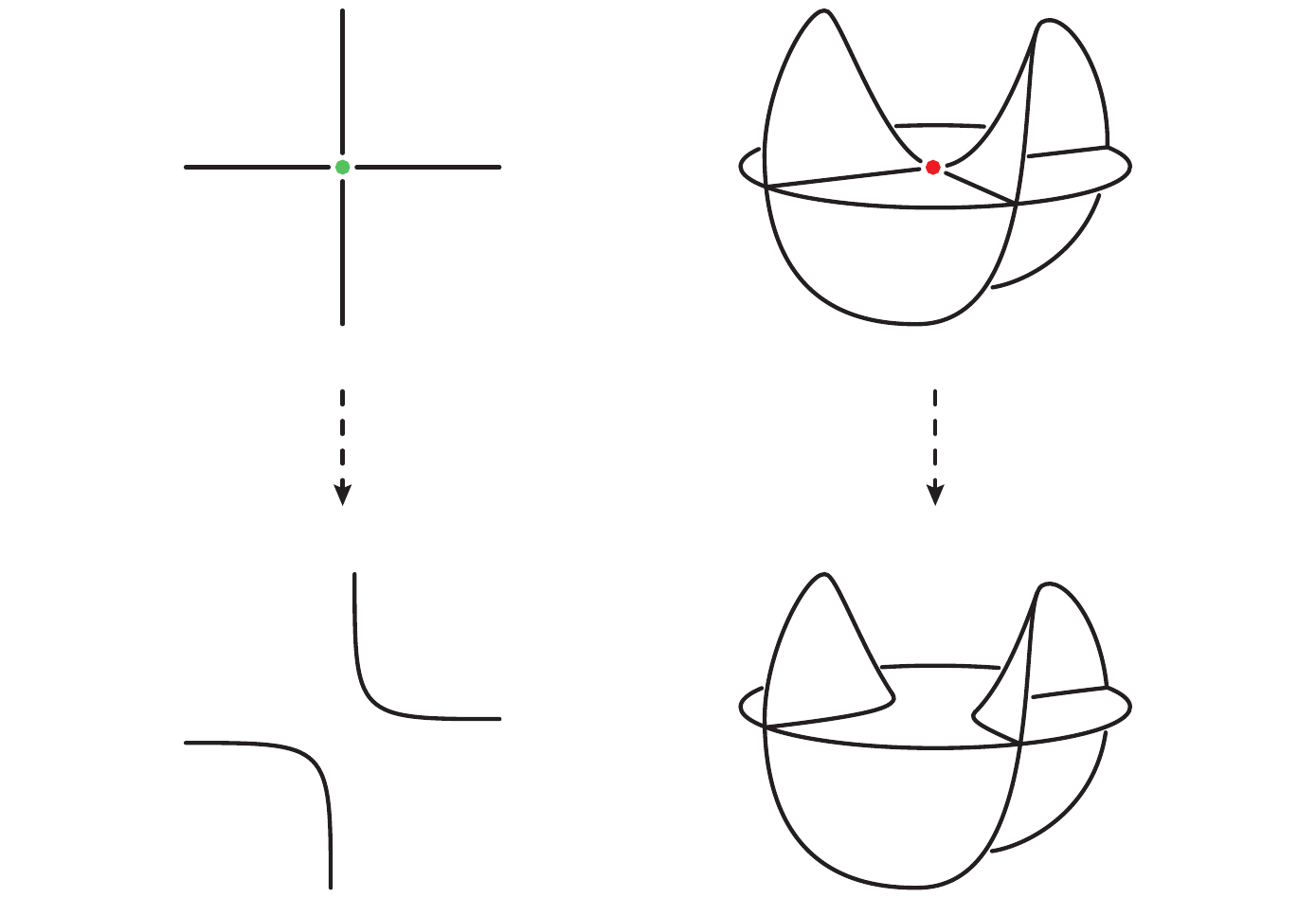}
\end{center}
\caption{The singularity $\{xy=0\}\subseteq \C^2$ and its Milnor fibre can be realized as the intersection of the branches of a bi-germ and its stable perturbation. The singular locus of $\{xy=0\}$ corresponds to the instability locus of the bi-germ.}
\label{figSaddle}
\end{figure}
\end{ex}

Our claim that the study of the Milnor fibration of hypersurfaces is
the same as the study of disentanglements of bi-germs of immersions 
also carries over to the vanishing topology in the following sense:

\begin{rem}\label{remCorrespondenceMilnorDisentanglement}
Let $M$ be the Milnor fibre of a hypersurface $X=V(g)$ and $Y_\delta =\im f_{\delta}$ be the disentanglement of the bi-germ associated to $g$, as in Example \ref{exConstructionBiGerm}. Then, 
there is an isomorphism \[\tilde H^i(M)\cong \tilde H^{i+1}(Y_\delta),\]
compatible with the monodromy actions on $M$ and $Y_\delta$.  In fact, $Y_\delta$ is homotopy equivalent to  the suspension of $M$ and we will come back to this isomorphism on  Example \ref{exSpectralSeqBiGermImmersion} and Remark \ref{remTwoPointSuspension} to illustrate the more general approach to the computation of the cohomology  of mappings. 
\end{rem}
 

\section{Main Results}

In this paper, we always assume that $f \colon (\C^n,S) \to (\C^p,0)$ is $\cK$-finite:

\begin{definition}
A map $f$ is \emph{$\cK$-finite} if the restriction $f\vert_{\Sigma(f)}$ is a finite map.
\end{definition}
This condition arises
naturally in the study of contact equivalence of singular map-germs
\cite[Section 4.4]{MNB20} and it is generally considered mild. What
appears here as its definition is usually regarded as a geometric
characterization \cite[Proposition 4.3]{MNB20}.
  Observe that, for $p> n$, the condition for a germ
  $f\colon(\C^n,S)\to(\C^p,0)$ to be $\cK$-finite is just that $f$ is
 finite.
 
 The $\cK$-finiteness condition is helpful in at least three ways: First of all, it entails that for 
 dimensions $p\leq n+1$ the discriminant
  $\Delta(f)$ is a hypersurface in $\C^p$ (Proposition
  \ref{propDimOfDelta}). 
 Second, the instability locus of  $\cK$-finite mappings is analytic
 (Proposition \ref{propKFiniteImpliesInstAnalytic}). Lastly,
 $\cK$-finiteness ensures the existence of stable unfoldings which
 are used in the definition of multiple point spaces in Section
 \ref{secMultPoints}.

\begin{thm}\label{KMB}
  Let $f\colon (\C^n,S) \to (\C^p,0)$ be a $\cK$-finite map-germ, with
  $p\leq n$, and let $d=\dim(\Inst(f))$. The reduced cohomology of the discriminant of any perturbation $f_\delta$
 satisfies \[\tilde H^q(\Delta(f_\delta))=0,\text{ for any }q\notin [p-1-d,p-1].\]
\end{thm}

The previous result also holds for disentanglements, as they are particular cases of perturbations.

\begin{thm}\label{KMA}
  Let $f\colon (\C^n,S) \to (\C^{n+1},0)$ be a finite map-germ and
  let $d=\dim(\Inst(f))$. The reduced cohomology of the image of any perturbation $f_\delta$
 satisfies  \[\tilde H^q(Y_\delta)=0\text{, for any }q\notin [n-d,n].\]
\end{thm}

In the case of germs $f\colon(\C^n,S)\to(\C^{n+1},0)$, apart from the
image $Y$, it is common to study the \emph{source double point space}
$D(f)\subseteq\C^n$. For a map $f\colon X\to Z$ between complex manifolds,
it consists of the set of points $x\in X$ such that
the germ $f\colon (X,f^{-1}(f(x)))\to(\im f,f(x))$ is not an isomorphism, that is, 
\[
  D(f)=\{x\in X\mid  f^{-1}(f(x))\neq\{x\}\}\cup\{x\in X\mid f \text{ is not immersive at $x$}\}.
\]
A point $x\in D(f)$ is called a \emph{double point} of $f$, even in the case that $f^{-1}(\{f(x)\})=\{x\}$.
Details on the analytic structure of $D(f)$ are found in Section
\ref{secProofsKMAKMB}. For now, it is enough to know the following: 
Whenever $f\colon
X\to Z$ is a finite mapping between complex manifolds with $\dim Z=\dim
X+1$ and $\dim (\Inst(f))<\dim X$, then $D(f)$ is a hypersurface
in $X$.

\begin{thm}\label{KMDoublePoints}
  Let $f\colon (\C^n,S) \to (\C^{n+1},0)$ be a finite map-germ and
  let $d=\dim(\Inst(f))$. The reduced cohomology of the source double point space of any perturbation $f_\delta$
 satisfies \[\tilde H^q(D(f_\delta))=0\text{, for any }q\notin [n-1-d,n-1].\]
\end{thm}

\begin{rem}\label{remDoublePointMilnorNumber} Just as the disentanglements and discriminants of Example \ref{exFinDet}, the source double points space $D(f_\delta)$ of a finitely determined map germ $f\colon (\C^n,S) \to (\C^{n+1},0)$ has the homotopy type of a bouquet of spheres of dimension $n-1$. This observation is due to R. Gim\'enez Conejero and J.J. Nu\~no-Ballesteros \cite{GN2020}. 
\end{rem}

\begin{ex}\label{exDoublePointsBiGerm} Let us come back to the bi-germ of immersion
\[f\colon U_1\sqcup U_2\to \C^{n+1}\] associated to a hypersurface $X\subseteq \C^n$, described in Example \ref{exConstructionBiGerm}. Since each of the branches $f\vert_{U_i}$ is injective and immersive, the space $D(f)$ is the preimage by $f$ of the intersection of the the two branches $Y_i=\im f\vert_{U_i}$. Therefore, the source double point space consist of two disjoint copies
\[D(f)=X\sqcup X.\]
The same argument shows that the double point space of the stable perturbation of $f$ is
 \[D(f_\delta)=M\sqcup M\]
where $M$ is the Milnor fibre of $X$. From this point of view the cohomological 
version of Kato's and
Matsumoto's connectivity result is a particular case of Theorem
\ref{KMDoublePoints}. The classical connectivity result due to Milnor 
appears as a particular case of Remark \ref{remDoublePointMilnorNumber}.
\end{ex}

\begin{cor} Let $f\colon (\C^n,0) \to (\C^{n+1},0)$ be a finite map-germ and assume that $\dim(\Inst(f))<n-1$. Then, the singular locus $\Sing Y_\delta$ of the image of any perturbation of $f$ is connected.
\begin{proof} Theorem \ref{KMDoublePoints} implies immediately that $D(f_\delta)$ is connected. Moreover, from the set-theoretical description of the source double point space given above, it follows that the singular locus of the image $Y$ of a finite and generically one-to-one map $f\colon X\to Z$ between complex manifolds is, as a set, the image of the double point space. In other words, finite and generically one-to-one maps satisfy $\Sing Y=f(D(f))$, which clearly implies our claim.  Now observe that, for dimensions $\dim Y>\dim X$, stable mappings are generically one-to-one. Therefore, mappings satisfying  $\dim(\Inst(f))<n-1$ (hence generically stable) are necessarily generically one-to-one as well.
\end{proof}
\end{cor}

In order to obtain cohomological connectivity results for germs $f\colon (\C^n,0) \to (\C^p,0)$ with arbitrary $p>n+1$, something stronger than $\cK$-finiteness is required: A map $f\colon U \to V$ between complex manifolds of dimensions $n=\dim U$ and $p=\dim V$
is \emph{dimensionally correct} if its multiple point spaces $D^k(f)$
are empty or have the minimal dimension $kn-(k-1)p$ (see Section
\ref{secMultPoints} for the definition of multiple points).

\begin{thm}\label{KM}
  Let $f\colon (\C^n,0) \to (\C^p,0)$  be a dimensionally correct
  map-germ and $p>n+1$. Then, for any perturbation $f_\delta$, all possibly
  non-trivial reduced cohomologies $\tilde H^q(Y_\delta)$ are concentrated
  in the degrees $q$ satisfying
  \[
    q= kn - (k-1) (p-1) - s,
  \]
  for  some 
  $1<k\leq \left\lfloor \frac{p}{p-n} \right\rfloor$ and $0\leq s \leq d=\dim \Inst(f)$.
 \end{thm}

As the following example shows, Theorem \ref{KM} does not hold if
the dimensionally correct hypothesis is omitted.

\begin{ex}\label{exNonDimCorrect}By adding a zero coordinate function to the cuspidal edge of Example \ref{exCuspEdge}, we obtain the germ $f'\colon(\C^2,0)\to(\C^4,0)$ given by
  \[
    (x,y)\mapsto(x,y^2,y^3,0).
  \] 
Similarly, adding a zero coordinate to the perturbation of Example \ref{exCuspEdge} gives a perturbation
  $f'_\delta$, given by
  \[
    (x,y)\mapsto(x,y^2,y^3+\delta y(x^2-\delta),0),
  \]
whose image is, obviously, isomorphic to that of the mentioned
example. This image has non-trivial cohomology in degree two.
If $f'$ was dimensionally correct, then Theorem \ref{KM} 
would allow for nontrivial vanishing cohomology only in degrees 
$0$ and $1$. Indeed, this condition on $f'$ is violated 
since the source double point space $D(f')$
contains the whole line $\{y=0\}$ ($f'$ is not immersive there) whereas, for
germs $(\C^2,0)\to(\C^4,0)$, the minimal dimension for double points is
$2\cdot 2-1\cdot 4=0$.
\end{ex}

\begin{rem} 
  A first version of Theorems \ref{KMA} and \ref{KM} was proven in
  \cite{PZ}, under the extra assumptions that $f_\delta$ is stable, that $f$ has
  corank one and---in the case of $p=n+1$---that $f$ is a dimensionally correct monogerm. Two examples of map-germs that only the new version can handle are $(x,y)\mapsto(x^2,y^2,x^3+y^3+xy)$ (which has corank two) and $(x,y)\mapsto(x,y^3,y^4)$ (which is not dimensionally correct, since it has a line of triple points, but satisfies the hypotheses of Theorem \ref{KMA}).
\end{rem}

Finally, let us return to the study of monodromy.
Let  $F$ be a one-parameter unfolding of a finite map-germ $f\colon (\C^n,0)\to (\C^p,0)$, with $p>n$. The projection from the image $\mathcal Y$ of a good representative of $F$ to the deformation parameter restricts to a locally trivial fibration
\[\mathcal Y^*\to D^*\] over the punctured disk. Since the disentanglement $Y_\delta$ is the fibre of the previous fibration, for each $i$ there is corresponding monodromy automorphism
\[h^i\colon H^i (Y_\delta,\C) \to H^i (Y_\delta,\C).\]
On one hand, the fact that $f$ is finite implies that $\mathcal Y$
is analytic, and thus it is clear that the eigenvalues of $h^i$ are roots
of the unity, cf. Remark \ref{rem:RootsOfUnity}. On the other hand $\mathcal Y$ is
not smooth, (in fact $\mathcal Y$ need not be a complete intersection
when $p>n+1$) and thus the classical result about the size of Jordan
blocks does not apply in our setting. As we will see, under the right
hypothesis the classical bound for the Jordan blocks and a spectral
sequence argument can be combined to bound the size of the Jordan blocks
of the monodromy of perturbations.

\begin{thm} \label{monodromy} 
Let $F$ be a stable one-parameter unfolding of a corank one map germ $f\colon (\C^n,0) \to (\C^p,0)$ with $p>n$. Then
the size of the Jordan blocks of $h^i$ are at most $i(i+1)/2$. 
Moreover, if $f$ has only isolated instability locus and $p>n+1$, the possibly non-trivial reduced cohomologies  are $\tilde{H}^{kn-(k-1)(p-1)}(Y_\delta)$, for $1< k \leq  \left\lfloor \frac{p}{p-n} \right\rfloor$, and the size of the corresponding Jordan block is at most $kn-(k-1)p+1$. 
\end{thm}  


\subsection{Some remarks on generality}In writing this paper, we decided to sacrifice some generality for the sake of clarity. The following remarks list some improvements that ended up being left behind: 

\begin{rem}
Throughout the text, the instability locus can be replaced by the \emph{topological instability locus}. Topological stability is defined by using homeomorphisms instead of diffeomorphisms in the definition of trivial unfolding. In general, the topological instability locus is smaller, giving rise to sharper cohomological connectivity bounds. 
\end{rem}

\begin{rem}
 Theorem \ref{KM} is stated for monogerms, but it applies to multigerms as well. 
 However, the involved representations of the symmetric group made 
the exposition more complicated and seemed off-topic.
\end{rem}

\begin{rem}
The ``dimensionally correct'' condition of Theorem \ref{KM} may be relaxed by introducing  \emph{strict} dimensional correctness and adding some extra hypotheses (counterintuitively, strict dimensional correctness is less demanding than dimensional correctness, see Definition \ref{defStrictlyDimCorrect}). The cohomological connectivity predicted by  Theorem \ref{KM} also holds in the following situations:
\begin{enumerate}
\item The unfolding $F$ underlying $f_\delta$ is strictly dimensionally correct.
\item The germ $f$ is strictly dimensionally correct, and $Y_\delta$ is a disentanglement (i.e. $f_\delta$ is a stable perturbation).
\end{enumerate}
The second situation is, in fact, a particular case of the first, because stable perturbations of strictly dimensionally correct germs arise only from strictly dimensionally correct unfoldings. The reason not to include this improvement is that it depends on the unfolding $F$, not only on $f$. Example \ref{exNonDimCorrect} shows that  perturbations $f_\delta$ of a strictly dimensionally correct map-germ $f$ need not satisfy the cohomological connectivity of Theorem \ref{KM}.
\end{rem}

We thank David Massey, who inspired us to use perverse sheaves as a tool
for understanding perturbations of map-germs. We also thank Juan Jos\'e
Nu\~no Ballesteros for useful discussions about $\cK$-finiteness. The 
third author would like to thank the Singularity group at BCAM in 
Bilbao for their kind hospitality.


\section{Preliminaries}\label{secPreliminaries}
Here we include definitions related to stability, properties involving $\cK$-finiteness and the basic notions of the theory of perverse sheaves which are used to show our results.
\subsection{Stability and $\cK$-finiteness}\label{secStability}

An unfolding $F\colon (\C^n\times\C^r,S\times\{0\})\to (\C^p,0)$ of a germ
$f$ is \emph{trivial} if there exist an unfolding $\Phi$ of $(\id_{\C^n},S)$
and an unfolding $\Psi$ of $(\id_{\C^p},0)$ 
such that the following
diagram commutes:
\[
\begin{tikzcd}
(\C^n\times\C^r,S\times\{0\}) \arrow{r}{F} \arrow{d}{\Phi}
&(\C^p\times \C^r,(0,0))\arrow{d}{\Psi}\\
(\C^n\times\C^r,S\times\{0\})\arrow{r}{f\times \id_{\C^r}}& (\C^p\times \C^r, (0,0)).
\end{tikzcd}
\]
Note that these conditions make $\Phi$ and $\Psi$ into germs of biholomorphism.

A germ $f$ is \emph{stable} if every unfolding of $f$ is trivial. A
$\cK$-finite map $f\colon M\to N$ is stable at $q\in N$ if the germ of $f$
at $S= \Sigma(f)\cap f^{-1}(q)$ is stable. We say that the map
$f$ is stable if it is $\cK$-finite and it is stable at every $q\in N$.

A one-parameter unfolding $F$ is called a \emph{stabilisation} of
$f$ if it admits a good representative such that for every non-zero
$\delta\in T$ the perturbations $f_\delta$ are stable. In this case, $f_\delta$ is called a \emph{stable
perturbation} of $f$, and $Y_\delta=\im f_\delta$ is called a
\emph{disentanglement} of $Y$.

The \emph{instability locus} of $f\colon M\to N$ is the support 
\[\Inst(f)=\supp \frac{\theta(f)}{T\cA_e (f)},\]
where $\theta(f)$ is the sheaf of vector fields along $f$, and $T\cA (f)$ is the extended $\cA$-tangent space of $f$, see \cite{WallFiniteDeterminacyOfSmoothMapGerms}. What gives its name to the instability locus is the following result from 
  \cite[Theorem 3.2]{MNB20}:
\begin{prop}
  \label{propInstContainsTheInstabilities}
  Let $f\colon M\to N$ be a $\cK$-finite map, let $q\in N$ and $S=f^{-1}(q)\cap \Sigma(f)$. If $q\notin\Inst(f)$, then the germ of $f$ at $S$ is stable.
\end{prop}

Introducing the instability locus as the support of a sheaf allows us to justify that it is an analytic space. The proof of the following proposition can be extracted from that of 
\cite[Proposition 4.2]{MNB20}:

\begin{prop}\label{propKFiniteImpliesInstAnalytic}
If $f$ is $\cK$-finite, then $\theta(f)/(T_e\cA (f))$ is a coherent $\cO_Y$-module. In particular, $\Inst(f)$ is analytic.
\end{prop}

We finish this subsection with a result which is well known, but whose proof we include for lack of
a reference.  

\begin{prop}\label{propDimOfDelta} 
  Let $f\colon X\to Z$  be a $\cK$-finite map between complex manifolds, 
  let $n=\dim X$ and $p=\dim Z$. If $p\leq n+1$,
  then the dimension of $\Delta(f)$ is $p-1$.
\end{prop}

\begin{proof}
  The case $p=n+1$ is clear, because
  then $\Sigma(f)=X$, the map is finite and $\Delta(f)=\im f$. We may thus
  assume $q\leq n$. Since $f$ is $\cK$-finite, $\Delta(f)$ is an analytic
  space of the same dimension as $\Sigma(f)$. On the one hand, 
  Sard's Theorem implies
  that $\Delta(f)$ is not all of $Z$, hence its dimension is at
  most $p-1$. On the other hand, $\Sigma(f)$ is defined as the vanishing 
  locus of $p$-minors of an $n\times p$-matrix with $p \leq n$. 
  The results in \cite{EagonNorthcott62} imply that the dimension of 
  any component of $\Delta(f)$ is greater or equal to 
  $n - (n - p + 1 ) = p-1$. 
\end{proof}

\subsection{Perverse sheaves}

In this subsection, we will summarize those parts of the 
machinery of perverse sheaves on complex analytic spaces which 
we shall need for the proof of our theorems. We follow 
the standard reference \cite{D2} and all the details 
can be found there. 

Throughout this section, $X$ stands for a complex analytic variety. 
For a sheaf $\mathscr F$ on $X$ its
\textit{sheaf cohomology groups} will be denoted by 
\[
  H^i(X; \mathscr F).
\]
We write $\C_X$ for the 
constant sheaf on $X$ associated to the field $\C$.
Recall that the cohomology of $\C_X$ is isomorphic to the singular cohomology 
with complex coefficients:
\[
  H^\bullet(X;\C_X) \cong H_\sing^\bullet(X).
\]

By $C(X)$ we denote the category of complexes of sheaves of $\C_X$-modules on  $X$:
\[
 \mathscr F^\bullet :\qquad \cdots \to \mathscr F^{i-1} \to \mathscr F^i \to \mathscr F^{i+1} \to \cdots
\]
From any such complex of sheaves we obtain the collection of 
\textit{cohomology sheaves}, denoted by 
\[
  \mathcal H^i( \mathscr F^\bullet), \quad i\in \Z.
\]
A morphism of complexes of sheaves $\varphi^\bullet \colon \mathscr F^\bullet \to \mathscr G^\bullet$
is called a \textit{quasi-isomorphism}, if the induced maps on the cohomology sheaves 
\[
  \varphi \colon \mathcal H^i(\mathscr F^\bullet) \to \mathcal H^i(\mathscr G^\bullet)
\]
is an isomorphism for all $i\in\Z$. 
By $\mathcal F^\bullet[d]$ we will denote the \textit{shift by} $d$ of the complex 
$\mathcal F$ which is given by the terms 
\[
  \left( \mathcal F^\bullet[d] \right)^k = \mathcal F^{k-d}
\]
for every $k$ together with the appropriately shifted differentials from $\mathcal F$. 

By $D(X)$ we denote the \textit{derived category} of sheaves of $\C_X$-modules on $X$ 
which is obtained from $C(X)$ by \textit{localizing} at the set of 
quasi-isomorphisms.
In particular this construction entails that two complexes of sheaves $\mathscr F^\bullet$
and $\mathscr G^\bullet$ are \textit{isomorphic} in $D(X)$ if and only if they are 
\textit{quasi-isomorphic} in $C(X)$. 

A complex of sheaves $\mathscr F^\bullet$ is in the \textit{bounded} derived 
category $D^b(X)$ if its nontrivial cohomology sheaves $\mathcal H^i(\mathscr F^\bullet)$ 
are confined to a bounded range of indices $i$.
For any such complex of sheaves $\mathscr F^\bullet \in D^b(X)$ we can 
define the \emph{support} as 
\[
  \supp\mathscr F^\bullet=\overline{\{x\in X\mid  
    \mathcal H^i( \mathscr F^\bullet)_x\neq 0 \text{ for some }i\in \Z\}}
\]
where we denote by $\mathcal H^i(\mathscr F^\bullet)_x$ the stalk of the cohomology 
sheaf $\mathcal H^i(\mathscr F^\bullet)$ at the point $x\in X$.

\medskip
A sheaf $ \mathscr F$ on $X$ can be regarded as a complex of sheaves concentrated in degree 0. 
This is a fully faithful embedding of the category of sheaves on $X$ into the 
bounded derived category, cf. \cite[Proposition 1.3.3 iii]{D2}.  This allows us to simplify notation 
and write $\mathscr F$ instead of $\mathscr F^\bullet$ for the complex 
of sheaves associated to a sheaf $\mathscr F$.


Under this embedding into the bounded derived category, 
classical sheaf cohomology reappears as 
\textit{hypercohomology}. More generally, for a continuous map 
$f \colon X \to Y$ of topological spaces and a sheaf $\mathscr F$ on 
$X$ one has 
\[
  R^i f_* \mathscr F \cong 
  \mathcal H^i( R f_* \mathscr F )
\]
where the left hand side denotes the $i$-th derived pushforward of a single sheaf 
and 
\[
  R f_* \colon D^b(X) \to D^b(Y)
\]
denotes the derived pushforward in the derived categories, cf. \cite[Section 2.3]{D2}.
For the special case of a projection $p \colon X \to \{pt\}$ to a point, one obtains 
\[
  H^i(X,\mathscr F) = 
  R^i p_* \mathscr F = 
  \mathcal H^i( R p_* \mathscr F )
  = \mathbb H^i( X, \mathscr F )
\]
where 
\[
  \mathbb H^i( X, - ) \colon D^b(X) \to \C\text{-Vect}
\]
is the $i$-th \textit{hypercohomology} functor.

Again for a map $f \colon X \to Y$ and a complex of sheaves $\mathscr F$ on $X$
the hypercohomology functors satisfy the following fundamental property:
\[
\mathbb H^i(Y, R f_*\mathscr F)=\mathbb H^i(X, \mathscr F),
\]
the \textit{Leray spectral sequence}, see \cite[Proposition 2.3.4]{D2}.

Similar translations of the above identifications between the category of 
$\C_X$-modules and the derived category exist for relative cohomology for 
pairs of spaces; see \cite{D2} for details.

\medskip

We shall say that a complex of sheaves $\mathscr F \in D^b(X)$ 
on a \textit{complex analytic} space $X$ 
is \emph{constructible} if there exists a complex analytic stratification of $X$ 
with locally closed \emph{complex analytic}
strata $S_\alpha \subset X$ such that 
\begin{enumerate}
  \item 
    The restriction of $\mathscr F$ 
    to each one of the $S_\alpha$ has locally constant cohomology sheaves, 
  \item 
    All stalks are finite dimensional $\C$-vector spaces. 
\end{enumerate}
The category of bounded constructible complexes of sheaves 
on $X$ will be denoted by $D_c^b(X)$. 

From now on, we will mostly be interested in a very special type of 
bounded complexes of constructible sheaves: 
the \textit{perverse sheaves} for the middle perversity. Note that a perverse sheaf need not be a sheaf, but a complex of sheaves.
It is our intention for the reader to accept ``being perverse'' 
as a good property which can be used without knowing its details. 
We proceed by stating the necessary properties and results for the 
intended usage.

\begin{prop} 
  \label{prp:stalkCohomologyBoundFromPerversity} \cite[Remark 5.1.19]{D2}
  Let $\sF$ be a perverse sheaf on a complex analytic variety $X$,  and let $d$ be the dimension of the support of $\sF$. For any point $x \in X$, the stalk cohomology groups satisfy
  \[
    \mathcal H^i(\sF)_x = 0 \textnormal{, for all } i\notin [-d, 0].
  \]
\end{prop}

Note that since by definition any perverse sheaf is constructible, 
its support is always a closed \emph{complex analytic} subset and it makes 
sense to speak of its (complex) dimension. 

\begin{prop}\label{propPerversityConstantSheafLCI} \cite[Theorem 5.1.20]{D2}
  \label{prp:PerversityOfConstantSheafOnCompleteIntersections}
  If $X$ is a locally complete intersection of complex dimension $d$, 
  then the shifted constant sheaf $\C_X[d]$ is perverse.
\end{prop}



Let $g: \cX \to \C$ be a holomorphic function defined on a complex
analytic variety and set $X=g^{-1}(0)$. Associated to $g$, there is the
\emph{vanishing cycle functor} \[\phi_g : D^b_c(\cX)\to D^b_c(X) .\]

 For any point $x\in X$ and any $\sF\in D^b_c(\cX)$, the
stalk cohomology of $\phi_g\mathscr F$ can be computed as follows
(see e.g. \cite[page 106 (4.1)]{D2}): 
\begin{equation} 
  \label{eqn:vanishingCycleFunctorStalk}
  \mathcal H^i (\phi_g \sF)_{x} = \mathbb H^{i+1}( B_{r}, B_{r}
  \cap g^{-1}(\{\delta\}); \sF),
\end{equation} 
where $B_{r}$ is a sufficiently small open ball in $\cX$
centered at $x$ and $\delta \in \C\setminus  \{ 0\}$ is small enough
with respect to $r$. The last result we need is essential, as it allows
to produce new perverse sheaves from old:

\begin{thm}
  \label{thm:vanishingCycleFunctorPreservesPerversity} \cite[Theorem 5.2.21]{D2} 
  The vanishing cycle functor shifted by one $\phi_g [-1]\colon D^b_c(\cX)\to
  D^b_c(X)$  takes perverse sheaves on $\cX$ to perverse sheaves on $X$.
\end{thm}

The ingredients we have introduced can be put together to give a simple proof of the cohomological connectivity of Milnor fibers: 

 Let $g$ be a suitable representative of a holomorphic map germ
 $(\C^{n+1},0) \to (\C,0)$ defined on some open subset $\mathcal X
 \subset \C^{n+1}$.  Take the perverse sheaf  $\C_{\mathcal X}[n+1]$
 and apply the shifted vanishing cycle functor associated to $g$ to
 obtain a perverse sheaf on $X = g^{-1}(0)$. The stalk cohomology at a point $x\in X$
 recovers the reduced cohomology of the Milnor fibre of the germ of $X$
 at $x$ as follows:
\begin{eqnarray*}
  \mathcal H^i (\phi_g  \C_{\mathcal{X}}[n])_{x}&=& \mathbb
   H^{i+1}(B_r, B_r \cap g^{-1}(\delta); \C_{\mathcal X}[n])\\
   &=& H^{i+n+1}(B_r, B_r \cap g^{-1}(\delta); \C_{\mathcal X} ) \\ &=&
  H_\sing^{i+n+1}(B_r, B_r \cap g^{-1}(\delta) ) \\ &=&
  \tilde H_\sing^{i+n}(B_r \cap g^{-1}(\delta) ).
\end{eqnarray*}

Since the Milnor fibre of $g$ at a regular point  of the function has trivial
reduced cohomology, we deduce that the support of $\phi_g
\C_{\mathcal{X}}[n]$ is contained in $\Sing X$. Letting $d=\dim(\Sing
X)$, Proposition \ref{prp:stalkCohomologyBoundFromPerversity} gives a bound on the stalk
cohomologies at the origin 
  \[
    \mathcal H^i\left(\phi_g \C_{\mathcal{X}}[n]\right)_0 \cong 
    \tilde H^{i+n}_{\text{sing}}(B_r \cap g^{-1}(\delta)) 
    = 0 \quad \text{ for all } i \notin [-d,0]
  \]
which turns out to be the desired cohomological connectivity
result for the Milnor fiber of $g$.

As the reader will see, the proofs of our results follow the same
pattern: Find an appropriate perverse sheaf and apply the vanishing cycle
functor associated to the projection to the parameter space. By virtue
of Proposition \ref{propDimOfDelta}, an estimate on the dimension of the
support of the sheaf will directly translate to a bound on the non-zero
degrees of the reduced cohomology of a nearby object.



\section{Proof of Theorems \ref{KMB}, \ref{KMA} and \ref{KMDoublePoints}}\label{secProofsKMAKMB}
Recall that, for $p=n+1$,  finiteness and $\cK$-finiteness are
the same and $Y=\Delta (f)$.  Consequently, Theorem \ref{KMB} and
Theorem \ref{KMA} can be considered two instances of the same result for $\cK$-finite germs
 $f\colon (\C^n,0)\to(\C^p,0)$ with $p\leq
n+1$. For a $\cK$-finite germ $f$, any unfolding $F$
is $\cK$-finite and, since $p\leq n+1$, the discriminant $\Delta(F)$ is a
hypersurface in $\C^{p}\times \C$, by Proposition \ref{propDimOfDelta}. In
particular, the shifted constant sheaf $\C_{\Delta(F)}[p]$  is perverse. Projecting on the unfolding parameter gives a family 
\[
  \Delta(F)\stackrel{\pi}\longrightarrow T,
\] 
and one checks the fiber over $\delta\in T$ to be $\Delta(f_\delta)$. Applying
the associated vanishing cycle functor  gives a perverse sheaf
$\phi_\pi \C_{\Delta(F)}[p-1]$ on $\Delta(f)$.
 
\begin{lem}\label{lemSuppVanishingCycleDiscriminant}
  The support of $\phi_\pi\C_{\Delta(F)}[p-1]$ is contained 
  in $\Inst(f)$.
\end{lem}

\begin{proof}
 

 Given a point  away from the instability locus $y \in \Delta(f) \setminus \Inst(f)$,  
  $S:=f^{-1}(y)\cap \Sigma(f)$ is a finite set because $f$ is $\cK$-finite. 
  By Proposition \ref{propInstContainsTheInstabilities},
  the germ $f \colon (U,S) \to (V,y)$
  of $f$ at $S$ is stable and thus the unfolding 
  \[
    F \colon (U \times T, S\times \{0\}) \to (V \times T,(y,0)) 
  \]
  is a trivial unfolding. Consequently, there exist unfoldings $\Phi$ and $\Psi$
  of the identity mappings $\id_U$ and $\id_V$, respectively, such that the following diagram commutes:
\[
\begin{tikzcd}
(U \times T, S \times \{0\} ) \arrow{r}{F} \arrow{d}{\Phi}
&(V \times T, (y,0))\arrow{d}{\Psi}\\
 (U,S) \times (T,0)\arrow{r}{f\times \id_{T}}&  (V,y) \times (T,0).
\end{tikzcd}
\]
The stalk at $y$ 
  of the sheaf of vanishing cycles of $\pi$ on $\Delta(F)$ is 

  \begin{eqnarray*}
    \mathcal H^{i}(\phi_\pi\C_{\Delta(F)}[p-1])_{y} &=& 
    \mathbb H^{i+1}(\Delta(F)\cap B_r, \pi^{-1}(\delta)\cap B_r; 
      \C_{\Delta(F)}[p-1]) 
    \\
    & \cong & 
    \mathbb H^{i+1}((\Delta(f)\times \C)\cap B_r, \Delta(f) \cap B_r; 
    \C_{\Delta(f)\times \C}[p-1])
    \\
    &=& H^{i+p}((\Delta(f) \times \C) \cap B_r, \Delta(f) \cap B_r; \C )\\
    &=& 0,
  \end{eqnarray*}
  which finishes the proof.
\end{proof}

\begin{proof}[Proof of Theorems \ref{KMB} and \ref{KMA}]
  From the inclusion $\supp(\phi_\pi\C_{\Delta(F)}[p-1])\subseteq \Inst(f)$
  and the hypothesis that $\dim \Inst(f)\leq d$, 
  applying Proposition \ref{prp:stalkCohomologyBoundFromPerversity} we obtain that
  \[
    \mathcal H^i(\phi_\pi\C_{\Delta(F)}[p-1])_0  = 0 \textnormal{ for } i\notin [-d, 0].
  \]
  Theorems \ref{KMB} and \ref{KMA} follow, because for a good representative
  the stalk at the origin is precisely
  \begin{eqnarray*}
    \mathcal H^i(\phi_\pi\C_{\Delta(F)}[p-1])_{0} 
    &=& \mathbb H^{i+1}(\Delta(F), \Delta(f_\delta); \C_{\Delta(F)}[p-1]) \\
    & = & \tilde H^{i+p-1}(\Delta(f_\delta)),
  \end{eqnarray*}
  where the last equality is due to the fact that $\Delta(F)$ is contractible.
\end{proof}
The proof of Theorem \ref{KMDoublePoints} is very similar and thus will only be sketched. Before that, we discuss some subtleties of the analytic structure of $D(f)$. In order to avoid pathologies related to unfoldings, the source double point space $D(f)$ is given an analytic structure which need not be reduced. For our purposes, we do not need to know the details of this construction \cite[Definition 2.2]{MararNunoPenafort}, but only the following two properties:

\begin{enumerate}
\item  For any unfolding $F=(f_t,t)\colon U\times T\to V\times T$, the fiber over $\delta\in T$ of the family 
\[D(F)\stackrel{\pi}{\longrightarrow}T\]
is the complex space $D(f_\delta)$. 
\item If $f\colon (\C^n,S)\to(\C^{n+1},0)$ is finite and generically one-to-one, then the space $D(f)$ is a hypersurface.

\end{enumerate}
 The first statement follows from  \cite[Lemma 4.2]{MararNunoPenafort}. For the second statement, the proof of \cite[Lemma 2.3]{MararNunoPenafort} shows that $D(f)$ is a  Cohen Macaulay subspace of $\C^n$ of codimension one, hence a hypersurface.
 
 \begin{proof}[Proof of Theorem \ref{KMDoublePoints}]
First of all, observe that the statement of Theorem \ref{KMDoublePoints} is trivial unless the dimension $d$ of the instability locus is smaller than $n-1$. Consequently, we may assume $f\colon (\C^n,S)\to (\C^{n+1},0)$ to be a finite map-germ with $\dim(\Inst(f))<n-1$.

The map $f$ is generically one-to-one, because stable mappings are generically one-to-one, and the conditions of finiteness and $\dim(\Inst(f))<n-1$ imply that the preimage of the instability locus is nowhere dense. 

Any unfolding $F$ of $f$ is also finite and generically one-to-one, hence $D(F)$ is a hypersurface in $\C^{n+1}$ and the sheaf $\C_{D(F)}[n]$ is perverse. The shifted vanishing cycle functor associated to the projection $\pi\colon D(F)\to T$ gives a perverse sheaf $\phi_\pi\C_{D(F)}[n-1]$ on $D(f)$. The same argument used in Lemma \ref{lemSuppVanishingCycleDiscriminant} shows the inclusion
\[\supp \phi_\pi\C_{D(F)}\subseteq f^{-1}(\Inst(f)),\]
where the dimension of $f^{-1}(\Inst(f))$ is equal to $d$, because $f$ is finite. Then the result follows 
from the computation of the stalk 
\[
 \mathcal H^i(\phi_\pi\C_{D(F)}[n-1])_{0}\cong \tilde H^{n-1+i}(\Delta(f_\delta)).\qedhere\]
\end{proof}

\section{Multiple point spaces}\label{secMultPoints}

For finite germs $f \colon (\C^n,0) \to (\C^p,0)$ with $p> n+1$, the 
image $Y$ is no longer a hypersurface, and may even fail to be a complete intersection, as shown by the well known twisted cubic:
\[
  (\C^2,0) \to (\C^4,0), \quad (s,t) \mapsto (s^3,s^2t,st^2,t^3).
\]
Since the constant sheaf on non complete 
intersections is not necessarily perverse, we cannot follow the same reasoning as in the proof 
of Theorems \ref{KMA} and \ref{KMB}. Instead, we study the 
cohomology of the disentanglement via the \textit{image computing 
spectral sequence} due to Goryunov and Mond \cite[Proposition 2.3]{GoryunovMond93}:
\[
  E^{p,q}_1 = H^q_{\Alt}(\mathscr D^{p+1}(f)) \Rightarrow H^{p+q}(Y).
\]
and the perversity of the sheaves 
\[
  \Alt R \varepsilon^k_* \C_{\mathscr D^k(f)}[kn -(k-1)p]
\]
on the image $Y$, discovered by Houston \cite[Theorem 2.9]{Houston00}. 
This involves the \textit{strict multiple point spaces} $\sD^k(f)$ 
and their \textit{alternating cohomology} which we will now discuss.

Unfortunately, there is no common agreement on the definition of multiple 
point spaces and different notions are in circulation. What we will refer to 
as the strict multiple points $\mathscr D^k(f)$ is the definition used by 
Goryunov, Mond \cite{GoryunovMond93}, and 
Houston \cite{Houston97}, \cite{Houston97Global}, \cite{Houston00}.
As Example \ref{exDimensionalCorrectness} shows, the strict multiple point spaces are 
badly behaved in deformations: They do not specialize to fibers. 
To remedy this fact, there is another, more subtle definition of 
multiple point spaces $D^k(f)$ due to Gaffney \cite{Ga83} which we will describe 
in Section \ref{secMultiplePointSpaces}. These are what we will refer to as 
simply the \textit{multiple point spaces of} $f$. 
Fortunately, the results about $\mathscr D^k(f)$ we want to use can be adapted to the spaces $D^k(f)$ without difficulties.

\subsection{The strict multiple point spaces $\mathscr D^k(f)$} 
\label{sec:StrictMultiplePointSpaces}

Let $f\colon X\to Z$ be a complex analytic map between holomorphic 
manifolds.  For $k\geq 1$, the $k$-th {\it strict multiple point space} of  $f$ is defined to be 
  \[
    \sD^k(f)= \textnormal{closure} 
    \{(x_{1},\dots,x_{k})\in X^k\mid f(x_i)=f(x_j), x_{i}\neq x_{j}
    \text{ for all }i\neq j\}.
  \]
Note that $\sD^1(f) =X$.

We recall the construction of the alternating complex due to Goryunov 
and Mond \cite{GoryunovMond93}. 
 $\sD^k(f)$ are symmetric in the sense that group $S_k$ acts on them
by permutating the points $x_1,\dots,x_k$. This action preserves 
the fibers of the maps
\[
  \varepsilon^k \colon \sD^k(f) \to Z, \quad 
  (x_1,\dots,x_k) \mapsto f(x_1) = \dots = f(x_k).
\]
As a consequence, there is an action of $S_k$ on the pushforward sheaf
$\varepsilon^k_*\C_{\sD^k(f)}$ on $Z$.
For each $\sigma
\in S_k$, we write the associated automorphism as \[
  \sigma^*\colon \varepsilon^k_*\C_{\sD^k(f)}\to
  \varepsilon^k_*\C_{\sD^k(f)}
\]
 and define
the \emph{alternating operator} 
$\Alt\colon \varepsilon^k_*\C_{\sD^k(f)}\to \varepsilon^k_*\C_{\sD^k(f)}$
by the formula
\[
  \Alt=\frac{1}{k!}\sum_{\sigma\in S_k}\sign(\sigma)\sigma^*.
\]
The image of $\Alt$ defines a subsheaf, which we denote by 
$\Alt \varepsilon^k_* \C_{\sD^k(f)}$.

\medskip
It is clear that the map $X^k \to X^{k-1}$ 
which forgets the $j$-th coordinate takes a $k$-multiple point to a
$(k-1)$-multiple point. Hence, we have maps
\[
  \varepsilon^{k,j}\colon \sD^k(f)\to \sD^{k-1}(f).
\]
These maps induce morphisms $\varepsilon^{k,j}_*\colon
\varepsilon^{k-1}_*\C_{\sD^{k-1}(f)}\to \varepsilon^{k}_*\C_{\sD^{k}(f)}$,
and one can see that there is a well defined differential
\[
  \partial^k\colon \Alt \varepsilon^{k-1}_*\C_{\sD^{k-1}(f)}\to 
  \Alt \varepsilon^k_*\C_{\sD^k(f)},
\]
of the form
\[
  \partial^k=\sum_{i=1}^k(-1)^i\varepsilon^{k,i}_*.
\]
One can also check the equality $\partial^{k+1}\circ\partial^k=0$,
so that we obtain a complex of sheaves $\left(\Alt
\varepsilon^{\bullet}_*\C_{\sD^{\bullet}(f)},\partial^\bullet\right)$, 
the \textit{alternating complex}.

\begin{prop} (\cite[Proposition 2.1]{GoryunovMond93})
  \label{prp:GoryunovMondAlternatingComplex}
  The augmented complex \[
    0 \to \C_Y \to 
    \Alt \varepsilon^{1}_*\C_{\sD^{1}(f)} \to
    \Alt \varepsilon^{2}_*\C_{\sD^{2}(f)} \to
    \cdots
  \]
  is exact.
\end{prop}

Goryunov and Mond have already argued that  one has an isomorphism
\begin{equation}
  H^i( Y, \Alt \varepsilon^k_* \C_{\sD^k(f)}) \cong H^i_{\Alt}(\sD^k(f))
  \label{eqn:GoryunovMondAlternatingSheavesAndCohomology}
\end{equation}
where the term on the right hand side is the \textit{alternating cohomology}
\[
  \left\{ [c] \in H^i_\sing(\sD^k(f)) : 
  \sigma^* [c] = \sign(\sigma)\cdot [c] \text{ for all } \sigma \in S_k
  \right\},
\]
a subspace of the singular cohomology of $\sD^k(f)$.
These considerations were taken to the derived category in 
\cite{Houston00}. Note that, since $\varepsilon^k$ is finite, 
the associated pushforward of sheaves is an exact functor 
and thus in particular
\[  \varepsilon^k_* \C_{\sD^k(f)}
  = R\varepsilon_*^k \C_{\sD^k(f)}
\]
as complexes of sheaves on $Y$ in the derived category. 
The same holds for their respective alternating subsheaves.

\begin{definition}\label{defStrictlyDimCorrect}
  Let $f\colon X\to Z$ be a complex analytic mapping between two complex
  analytic manifolds with $n=\dim X$ and $p=\dim Z$.
  Then $f$ is called {\it strictly dimensionally correct} if for all $k\geq 2$
  the strict multiple point space $\mathscr D^k(f)$ is either empty, 
  or has dimension $ k n-(k-1) p$.
\end{definition}

With these notations gathered
we may cite the key result due to Houston, cf. \cite[Theorem 2.9]{Houston00},
slightly adapted to our setup:

\begin{thm} \label{thm:HoustonPervesity}
  Suppose $f \colon X \to Z$ is a strictly dimensionally correct complex analytic 
  map between complex manifolds of dimensions $n = \dim X < \dim Z = p$. Then 
  $ \Alt \varepsilon^k_*\C_{\sD^k(f)}[kn-(k-1)p] $ is a perverse sheaf.
\end{thm}

\subsection{The multiple point spaces $D^k(f)$} \label{secMultiplePointSpaces} 

As we already mentioned before, the strict multiple point spaces discussed 
in the previous section are not well behaved in families. This is illustrated 
by the following example.

\begin{ex}\label{exStrictDimCorrect}
  \label{exDimensionalCorrectness}
  The cuspidal edge $f \colon (x,y)\mapsto(x,y^2,y^3)$ of Example \ref{exCuspEdge}
  and the similar germ $f' \colon (x,y)\mapsto (x,y^2,y^3,0)$ of Example
  \ref{exNonDimCorrect} are both strictly dimensionally correct because 
  the maps are injective and therefore the multiple point spaces are empty.

  However, the map $F' \colon (x,y,t)\mapsto
  (x,y^2,y^3+ty(x^2-1),0,t)$ from the unfolding in 
  Example \ref{exNonDimCorrect} is \emph{not} strictly dimensionally
  correct: Its multiple point spaces $\sD^1(F)$ have dimension two while 
  the expected dimension is one. 
  
  Without the dummy variable, i.e. for the map $f$ as in 
  Example \ref{exCuspEdge}, not only $f$, but also the unfolding 
   \[
    F \colon (x,y,t)\mapsto(x,y^2,y^3+t\cdot y(x^2-t))
  \]
  is strictly dimensionally correct with expected dimension two for 
  the double point spaces. In this example we encounter another problem:
  Observe that the strict double points of $f$ are empty
  while those of $F$ are not. This shows the failure of specialization
  \[\sD^2(F)\cap\{t=0\}\neq\sD^2(f)\]
  and thereby illustrates the necessity to use Gaffney's multiple 
  point spaces. 
\end{ex}

The pathological behaviour exhibited in this example can be avoided by taking
a different definition of multiple points due to Gaffney \cite{Ga83}. Our
exposition follows \cite{Nuno-Ballesteros2015On-multiple-poi}, where
the reader can find proofs and details omitted here.

\medskip

Let $f\colon(\C^n,S)\to(\C^p,0)$ be a finite multi-germ. Since $f$
is finite, hence $\cK$-finite, it admits a stable unfolding  \cite[Theorem 7.2]{MNB20}
\[
  F=(f_t,t)\colon(\C^n \times \C^r,S\times\{0\})\to(\C^p \times \C^r,(0,0))
\]
where $t$ is the coordinate of the parameter space $(\C^r,0)$.
It is clear that a point in $\sD^k(F)$ has $t_i=t_j$, for all $1\leq
i,j\leq k$. Therefore $\sD^k(F)$ can be embedded in $(\C^n)^k\times \C^r$
rather than in $(\C^n \times \C^r)^k$. The
multiple point spaces of $f$ are defined as
\[
  D^k(f):=\sD^k(F)\cap\{t=0\}.
\]

The space $D^k(f)$ does not depend on the chosen stable unfolding $F$.
 Multiple point spaces of finite mappings between complex manifolds are
      defined by patching the multiple point spaces of their corresponding
      multi-germs.
      
To compare $\mathscr D^k(f)$ to $D^k(f)$, we write the \emph{the $k$-th fat diagonal of $X$} as
\[ 
  \Delta^k = \{ (x_1,\dots,x_k) \in X^k \mid x_i = x_j \text{ for some } i \neq j \}.
\]
\begin{prop}{(Properties of the multiple point spaces)}
  \label{prp:PropertiesOfTheMultiplePointSpaces}
Let $f\colon X\to Z$ be a finite mapping between complex manifolds. Then
  \label{rem:PropertiesOfMultiplePoints}
  \begin{enumerate}
    \item  the spaces $D^k(f)$ 
      and $\sD^k(f)$ satisfy the relation
      \[\sD^k(f)=\overline{D^k(f)\setminus \Delta^k};\]  

          \item 
	    if $f$ is stable, then $D^k(f)=\sD^k(f)$; 
  \item 
	 unlike $\mathscr D^k(f)$, the spaces $D^k(f)$ behave well under
      deformations in the sense that, for any unfolding $F\colon X\times T\to
      Z\times T$,    the family $\pi\colon D^k(F)\to T$ satisfies
      \[
	D^k(f_{\delta})=\pi^{-1}(\delta),\text{ for any }\delta \in T.
      \]

  \end{enumerate}
\end{prop}
Just as in the case of $\sD^k(f)$, the spaces $D^k(f)$ are $S_k$-invariant spaces endowed with finite maps $\varepsilon^{k,j}\colon D^k(f)\to D^{k-1}(f)$, giving rise to a complex 
\[
  \left(\Alt
\varepsilon^{\bullet}_*\C_{D^{\bullet}(f)},\partial^\bullet\right).
\]
We wish to replace $\sD^k(f)$ by our favourite $D^k(f)$ in the study
of $\C_Y$. While the spaces $D^k(f)$ and $\sD^k(f)$ and hence also 
the sheaves $\varepsilon_*^k \C_{D^k(f)}$ and $\varepsilon^k_* \C_{\sD^k(f)}$ 
may differ, the corresponding \textit{alternating} sheaves do not:

\begin{lem} \label{lem:AlternatingSheavesAreEqual} 
  For any finite map $f$ between complex manifolds, the complexes $\left(\Alt
\varepsilon^{\bullet}_*\C_{D^{\bullet}(f)},\partial^\bullet\right)$ and $\left(\Alt
\varepsilon^{\bullet}_*\C_{\sD^{\bullet}(f)},\partial^\bullet\right)$ are equal.
\end{lem}

Note that with this lemma we may replace $\Alt \varepsilon^k_* \C_{\sD^k(f)}$ by 
$\Alt \varepsilon_*^k \C_{D^k(f)}$ in Proposition \ref{prp:GoryunovMondAlternatingComplex}. 
The idea behind the proof of Lemma \ref{lem:AlternatingSheavesAreEqual} is not 
new and appears already in \cite{GoryunovMond93}, as well as \cite{Houston00} and
\cite{Houston97Global}. Due to the lack of a citeable reference, we 
give a self contained argument.

\begin{proof}
  We verify that the stalks of both sheaves coincide at an arbitrary point 
  $y \in Y$. Fix one such point and let 
  \[
    P := \left( \varepsilon^k \right)^{-1}(y) \subset D^k(f)
  \]
  be its preimage under the finite map $\varepsilon^k \colon D^k(f) \to Y$.
  Set $P' := \sD^k(f) \cap P$ to be the preimage in the subspace 
  $\sD^k(f)$. According to Proposition \ref{prp:PropertiesOfTheMultiplePointSpaces}, 
  every point $x \in P \setminus P'$ has to be contained in the fat diagonal
  $\Delta^k$. Now the pushforward of $\C_{D^k(f)}$ along $\varepsilon^k$ 
  decomposes as 
  \begin{eqnarray*}
    \varepsilon_*^k \C_{D^k(f)} &=& 
    \left( \bigoplus_{x \in P'} \C_{D^k(f),x} \right) 
    \oplus
    \left( \bigoplus_{x \notin P'} \C_{D^k(f),x} \right).
  \end{eqnarray*}
  The alternating operator clearly respects this decomposition 
  since $P'$ is itself an $S_k$-invariant subspace. 
  Note that the first summand is nothing but the stalk of 
  $\varepsilon_*^k \C_{\sD^k(f)}$ by definition of $P'$. Thus, 
  we may conclude the proof by showing that the second summand 
  has no alternating part.

  To see the latter, let $x = (x_1,\dots,x_k) \in P \setminus P'$ 
  be an arbitrary point and $c \in \C_{D^k(f),x}$ an element of 
  its associated stalk. Since $x$ has to be a diagonal point, 
  there exists one pair of indices $0< i<j \leq k$ such that 
  $x_i = x_j$. Let $\tau \in S_k$ be the corresponding transposition. 
  Then we have 
  \[
    - \Alt c = \tau* \Alt c = \Alt (\tau*c) = \Alt c,
  \]
  because $\Alt c$ is alternating and $\tau$ takes the point $x$ to 
  itself. Thus, $\Alt c = 0$ for every element $c$ of any stalk 
  of $\C_{D^k(f)}$ at points outside $\sD^k(f)$.
\end{proof}

\begin{rem}\label{remHypercohomAltSheavesDk}
  Parallel to (\ref{eqn:GoryunovMondAlternatingSheavesAndCohomology}) we have the equality
  \[
    H^i(Y,\Alt\varepsilon^k_*\C_{D^k(f)})=H^i_{\Alt}(D^k(f),\C).
  \]
  This shows indirectly that $\sD^k(f)$ and $D^k(f)$ have the
  same alternating cohomology, thanks to the equality $\left(\Alt
  \varepsilon^{\bullet}_*\C_{D^{\bullet}(f)},\partial^\bullet\right)=\left(\Alt
  \varepsilon^{\bullet}_*\C_{\sD^{\bullet}(f)},\partial^\bullet\right)$ 
  proved in Lemma \ref{lem:AlternatingSheavesAreEqual}.
\end{rem}

As in the case of $\sD^k(f)$, if $f\colon X\to Z$ is a mapping between two complex
  analytic manifolds with $n=\dim X$ and $p=\dim Z$, then every irreducible component of $D^k(f)$ has dimension at least $kn-(k-1)p$. The use of $D^k(f)$ rather than $\sD^k(f)$ calls for the following adaptation:

\begin{definition}
  A finite map-germ $f\colon(\C^n,S)\to(\C^p,0)$ with $p>n$ is
  \emph{dimensionally correct} if for all $k\geq 2$ the multiple point
  space $D^k(f)$ is either empty or has dimension $k n-(k-1) p$. 
\end{definition}

It is clear that every unfolding $F$ of a finite map-germ $f$ must be finite as well. Moreover, any stable unfolding of $F$ is a stable unfolding of $f$ as well. From this observation, one concludes easily the following result:

\begin{prop}\label{propDimCorrectUnfolding}If a map-germ $f$ is dimensionally correct, then every unfolding $F$ of $f$ is dimensionally correct.
\end{prop}

\begin{rem}
  \label{rem:DiscussionOfDimensionalCorrectness} The previous assertion
  does not hold if one replaces the property of being dimensionally
  correct by that of being strictly dimensionally correct (see Example
  \ref{exStrictDimCorrect}).  Being dimensionally correct implies being
  strictly dimensionally correct, but the converse is not true, as shown
  again by the germ $(x,y)\mapsto(x,y^2,y^3,0)$.
\end{rem}

\subsection{The image computing spectral sequence} 

Let $f\colon (\C^n,0)\to (\C^p,0)$ be a finite map-germ with $p>n$ and let
$F$ be a one-parameter unfolding of $f$. We wish to compute the 
reduced cohomology of the disentanglement $Y_\delta$ of a good representative 
of $F$. Goryunov and Mond have shown \cite{GoryunovMond93} how 
to do this in terms of the alternating complex. We are basically going 
to rephrase their argument in the derived category so that we can bring the 
perversity of the alternating sheaves, Theorem \ref{thm:HoustonPervesity}, into the 
picture. 

In the derived category, Proposition \ref{prp:GoryunovMondAlternatingComplex} 
and Lemma \ref{lem:AlternatingSheavesAreEqual} 
assert that the constant sheaf $\C_{\mathcal Y}$ on the image $\mathcal Y$ of the unfolding $F$
is quasi-isomorphic to the alternating complex 
$\left( \Alt \varepsilon^\bullet_* \C_{D^k(F)}, \partial^\bullet \right)$ 
and we may thus replace $\C_{\mathcal Y}$ by this complex.

If we let $\pi \colon \mathcal Y \to T$ be the projection onto the unfolding 
parameter, then
the reduced cohomology of the disentanglement $Y_\delta$ 
are the vanishing cycles of $\pi$ at the origin:
\begin{eqnarray*}
  \tilde H^{i}(Y_\delta) &\cong& \mathcal H^i(  \phi_\pi\left( \C_{\mathcal Y}\right))_0 \\
  &\cong&  \mathcal H^i(\left(\phi_\pi (\Alt \varepsilon^\bullet_* \C_{D^\bullet(F)}) \right)_0.
\end{eqnarray*}

\medskip

The nearby and the vanishing cycles of
$\pi \colon \mathcal Y \to T$ are computed
using the well known diagram
\begin{equation}
  \xymatrix{
    Y_0 \ar@{^{(}->}[r]^i \ar[d] & 
    \mathcal Y \ar[d]^\pi & 
    \mathcal Y \setminus Y_0 \ar@{_{(}->}[l]_j  \ar[d]&
    Y_\delta \times S \ar[l]_\rho \ar[d] \\
    \{0\} \ar@{^{(}->}[r] & 
    T &
    T^* \ar@{_{(}->}[l] &
    S \ar[l]^{\exp}
  }
  \label{eqn:DiagramVanishingCycles}
\end{equation}
for a good representative $F$ of a $1$-parameter family and its image
$\mathcal Y$. Here, $\exp \colon S \to T^*$ is the universal cover 
of the punctured unit disc and 
$Y_\delta \times S$ the trivialized fiber product over the infinite strip $S$, 
cf. \cite[Chapter 4.2]{D2}.

\begin{lem}
  For a good representative $F$ of a $1$-parameter unfolding of 
  a $\mathcal K$-finite map 
  $f \colon (\C^n,0) \to (\C^p,0)$ with $n<p$
  there is a spectral sequence with first page 
  \begin{eqnarray*}
    E^{i,j}_1 &=& \mathcal H^i\left(
    \phi_\pi\left(\Alt \varepsilon^{j+1}_* \C_{D^{j+1}(F)}\right)
    \right)_0
  \end{eqnarray*}
  converging to the reduced cohomology of the disentanglement $Y_\delta$ of $f$ 
  \[
    \tilde H^{i+j}( Y_\delta) = \mathcal H^{i+j}(\phi_\pi \C_{\mathcal Y})_0.
  \]
  Here $\pi \colon \mathcal Y \to T$ is the projection of the image $\mathcal Y$ of 
  $F$ to the deformation parameter.
  \label{lem:SpectralSequenceForAlternatingCohomology}
\end{lem}

\begin{proof}
  The standard procedure to obtain a quasi-isomorphism of the complex 
  $\Alt \varepsilon^\bullet_* \C_{D^\bullet(F)}$ 
  with a complex of injectives is to construct a double complex of injectives 
  $I^{\bullet,\bullet}$ where each column $I^{\bullet,j}$ resolves the sheaf 
  $\Alt \varepsilon_*^{j+1} \C_{D^{j+1}(F)}$, cf. \cite[Section 2]{GoryunovMond93}.
  Then the original complex is quasi-isomorphic to the total complex $\Tot(I^{\bullet,\bullet})$ 
  of this double complex $I^{\bullet,\bullet}$. 

  The \textit{nearby cycles} of $\pi$ are defined as 
  \[
    \psi_\pi \left( \Alt \varepsilon_*^\bullet \C_{D^\bullet(F)} \right) 
    :=
    i^{-1} R(\rho\circ j)_* (\rho \circ j)^{-1} 
    \left( \Alt \varepsilon_*^\bullet \C_{D^\bullet(F)} \right).
  \]
  This is a composition of three functors. The first one, 
  $(\rho \circ j)^{-1}$, takes injective sheaves on $\mathcal Y$ 
  to injective sheaves on the open subset $\mathcal Y \setminus Y_0$, 
  see \cite[Proposition 2.4.1]{KashiwaraShapira} and the third one, 
  $i^{-1}$, is an exact functor. To compute the nearby cycles 
  of $\Alt \varepsilon_*^\bullet \C_{D^\bullet(F)}$---which involves the derived pushforward in the middle---we may therefore apply the functor $i^{-1} (\rho \circ j)_* (\rho \circ j)^{-1}$
  to $\Tot(I^{\bullet,\bullet})$. 
  Note that applying this functor 
  to a seperate column $I^{\bullet,j}$ yields the nearby cycles of the single sheaf 
  $\Alt \varepsilon_*^{j+1} \C_{D^{j+1}(F)}$ considered as a complex of sheaves 
  concentrated in a single degree.

  The \textit{vanishing cycles}  $\phi_{\pi}\C_{\mathcal Y} $
  of $\pi$ are defined as the mapping 
  cone of the comparison morphism $c$ 
  \[
    \xymatrix{
      i^{-1} \C_{\mathcal Y} \ar[r]^c &
      \psi_{\pi} \C_{\mathcal Y} \ar[r] &
      \phi_\pi \C_{\mathcal Y} \ar[r]^>>>>>{[+1]} &
      .
    }
  \]
  Following Definition \ref{defGoodRep}),
  the image $\mathcal Y$ of a good representative is connected, hence the 
  stalk of $i^{-1}\C_{\mathcal Y}$ 
  at the origin is simply the vector space $\C$, concentrated in degree zero. It is easy to 
  check that on this single nontrivial term the comparison morphism $c$ 
  is always an inclusion which makes the vanishing cycles coincide 
  with the reduced cohomology of the disentanglement, cf. (\ref{eqn:vanishingCycleFunctorStalk}). 
  
  We may again replace $\C_{\mathcal Y}$ by the 
  alternating complex $\Alt \varepsilon_*^\bullet \C_{D^\bullet(F)}$ and 
  subsequently by $\Tot(I^{\bullet,\bullet})$ in the distinguished triangle of 
  the comparison morphism:
  \[
    \xymatrix{
      i^{-1}\Tot(I^{\bullet,\bullet})\ar[r]^<<<<c &
      i^{-1} (\rho\circ j)_* (\rho \circ j)^{-1}\Tot(I^{\bullet,\bullet}) \ar[r] &
      \phi_\pi\Tot(I^{\bullet,\bullet})\ar[r]^>>>>>{[+1]} &
      .
    }
  \]
  Now note that the functor $i^{-1} (\rho\circ j)_* (\rho \circ j)^{-1}$ 
  commutes with taking the total complex in a natural way. 
  Thus we may identify the distinguished triangle of the comparison 
  morphism $c$ with the total complex of the following ``comparison morphism of 
  double complexes''
  \[
    \xymatrix{
      i^{-1} I^{\bullet,\bullet}\ar[r]^<<<<{\tilde c} &
      i^{-1} (\rho\circ j)_* (\rho \circ j)^{-1}I^{\bullet,\bullet} \ar[r] &
      \text{Cone}(\tilde c) \ar[r]^>>>>>{[+1]}&
    }
  \]
  which turns the vanishing cycles $\phi_\pi\Tot(I^{\bullet,\bullet})$ 
  into the total complex of the double complex $\text{Cone}(\tilde c)$.

  The spectral sequence in question is now the spectral sequence of 
  this double complex $\text{Cone}(\tilde c)$. A column-wise 
  inspection of this double complex reveals that
  we indeed find the vanishing cycles of the alternating sheaves on the 
  first page.
\end{proof}

\begin{rem}\label{remIsomInftyPage}
  From the fact that the spaces $D^k(F)$ are empty for $k$ big enough, it follows that the spectral sequence collapses at a certain page. Since we chose a field for the coefficients in 
  cohomology, the infinity page determines the cohomology of the total complex. To be precise,  there is an isomorphism of vector spaces 
\[\mathcal H^\ell (\phi_\pi\C_{\mathcal Y})_0\cong \bigoplus_i E^{i,\ell-i}_\infty.\]
\end{rem}

\begin{ex}\label{exSpectralSeqBiGermImmersion}Here we come back to the immersion bi-germ  $f\colon (\C^n,S)\to(\C^{n+1},0)$  associated to a hypersurface $X=V(g)$, introduced in Example \ref{exDoublePointsBiGerm}, where we find a particularly simple spectral sequence.
  First of all, there are no triple or higher multiplicity points. This is because the stable unfolding $F$ of $f$ (or any unfolding of $f$ for that matter) is also a bi-germ of an immersion, hence it cannot map more than two points to one. This gives $\sD^k(F)=\emptyset$, and thus $D^k(f)=\emptyset$, for all $k\geq 3$ . Consequently, the nonzero terms of the first page 
  $E_1^{\bullet,\bullet}$ are concentrated in the second column, corresponding to the alternating cohomologies $H^i_{\Alt}(D^2(f_\delta))$. 

The fact that there are no triple points implies also that the maps $D^2(F)\to D(F)$ and $D^2(f_\delta)\to D(f_\delta)$ are homeomorphisms. Recall also from Example \ref{exDoublePointsBiGerm}
that $D(f_\delta)=M_1\sqcup M_2$, where $M_1$ and $M_2$ are copies of the Milnor fibre of $X=V(g)$. The homeomorphism $D^2(f_\delta)\to D(f_\delta)$ transforms the action of the generator $\sigma\in S_2$ into the map taking a point $x\in M_1$ to the same point in $M_2$, and vice-versa. In particular, after identifying $D^2(f_\delta)$ with $M_1\sqcup M_2$, we can choose a system of generators of the cohomology $H^i(D^2(f_\delta))$ consisting of some cocycles $c_i$  generating the cohomology of $M_1$ and their corresponding cocycles $\sigma\cdot c_i$, which generate the cohomology of $M_2$. It now follows immediately that the alternating part is generated by the cocycles $c_i-\sigma\cdot c_i\in H^i_{\Alt}(D^2(f_\delta))$. The  linear map extending $c_i\mapsto c_i-\sigma\cdot c_i$ gives an isomorphism 
\[H^i(M)\cong H^i_{\Alt}(D^2(f)).\]

Finally, since $E_1^{\bullet,\bullet}$ contains a single nonzero column, the spectral sequence collapses at page one, that is, $E_\infty^{\bullet,\bullet}=E_1^{\bullet,\bullet}$. Now, the isomorphism $\mathcal H^\ell(\phi_\pi\C_{\mathcal Y})_0\cong \bigoplus_i E^{i,\ell-i}_\infty$ takes the  form
\[\tilde H^i(M)\cong \tilde H^{i+1}(Y_\delta),\]
as Remark \ref{remCorrespondenceMilnorDisentanglement} claimed.
\end{ex}

\begin{rem}\label{remTwoPointSuspension}The previous example has a more visual version in homology (Figure \ref {BigermSuspension} depicts the case of $g=x^2+y^2$ which, after a change of coordinates, is the same as in Figure \ref{figSaddle}, but with a real picture better suited for the present discussion). We know that $D(f_\delta)$ consists of two copies  $M_1$ and $M_2$ of $M$, contained in two copies $U_1$ and $U_2$ of an open ball $U\in \C^n$. Now let $c$ be a cycle in $M$ and let $c_i$ be the corresponding copies in $U_i$ (the two circles one the left side in Figure \ref{figSaddle}).  Since $U$ is contractible, there are chains $a_i$ in $U_i$ with boundary $\delta a_i=c_i$(the green and blue disks).  Observe that $\partial a_i$ is supported on $M_i$, which is the double point space of $f_\delta$, and that $f_\delta$ glues $a_1$ and $a_2$ along the boundaries $\partial a_i$. After changing the sign of $a_2$ if necessary, we observe that $f_*(a_1+a_2)$ is an $(i+1)$-dimensional cycle on $Y_\delta$ (the blue and green cycle on the right side). The desired isomorphism is determined by $c\mapsto f_*(a_1+a_2)$.
\end{rem}

 \begin{figure}
\begin{center}
\includegraphics[scale=0.9]{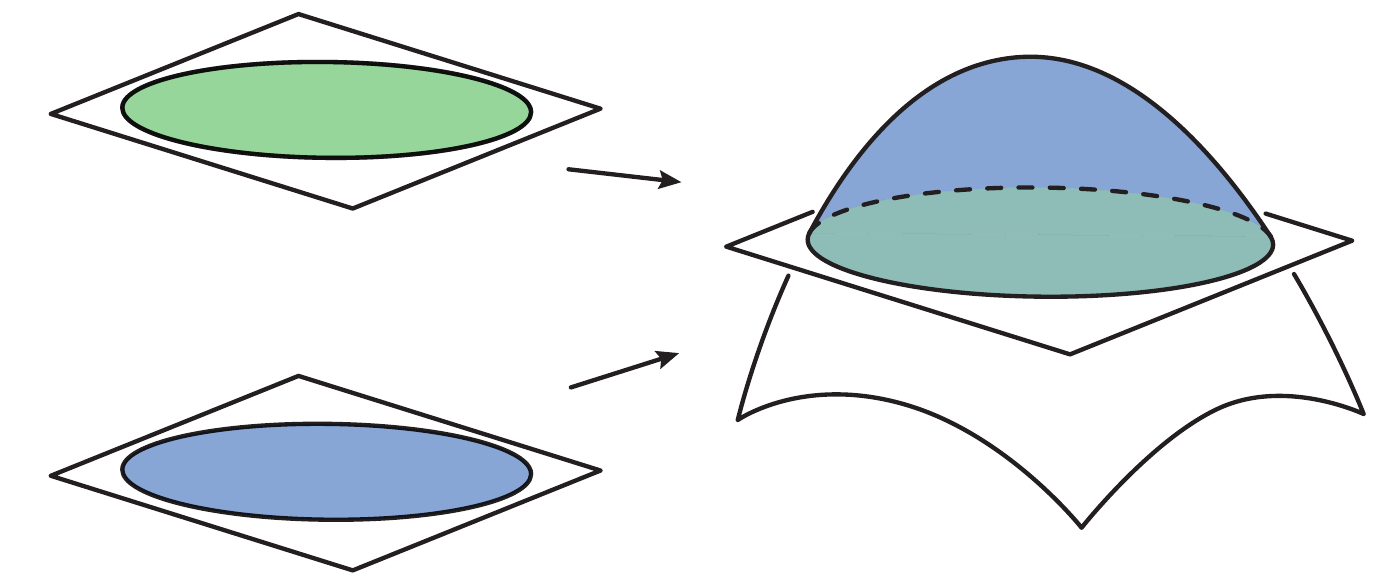}
\end{center}
\caption{The disentanglement of a bi-germ of immersion has the homology of a two-point suspension.}
\label{BigermSuspension}
\end{figure}

\section{Proof of Theorem \ref{KM}}

In this section we will use the notation of Theorem \ref{KM} and 
impose the hypothesis found there,
that is, $f\colon (\C^n,0) \to (\C^p,0)$ is a dimensionally correct
map-germ with $p>n$ and instability locus $\Inst(f)$ of dimension $d$,
$F\colon W \to V\times T$ is a good representative of a one-parameter unfolding of $f$ 
and $\mathcal Y=\im F$ is the image of $F$. 

We intend to control the stalk cohomology of $\phi_{\pi}\C_{\mathcal Y}$, 
i.e. the reduced cohomology of the disentanglement, by means of the first page 
of the spectral sequence in Lemma \ref{lem:SpectralSequenceForAlternatingCohomology}. 
Thus, we have to bound the vanishing cohomology of the alternating sheaves 
$\Alt \varepsilon_*^{j+1} \C_{D^{j+1}(F)}$ on the image and we will do so by 
exploiting the perversity of their respective shifts and estimating 
the dimension of the support of their vanishing cycles. 

\begin{lem} \label{lem:EstimateSupportVanishingCycles} 
  The support of $\phi_\pi \left(\Alt \varepsilon^k_*\C_{D^k(F)}\right)$ 
  is contained in $\Inst(f)$.
\end{lem}

\begin{proof}Take a point away from the instability locus
  $y\in Y\setminus \Inst(f)$ and let $S=f^{-1}(y)$. Since $f$
is stable at $y$, the unfolding $(F,S\times\{0\})$ of the germ
$(f,S)$ is trivial. Therefore, there exist an unfolding $\Phi$ of
$(\id_{\C^n},S)$ and an unfolding $\Psi$ of $(\id_{\C^p},0)$, such that
\[(F,S\times\{0\})=\Psi^{-1}\circ ((f,S)\times\id_{(T,0)})\circ\Phi.\]
We choose representatives $F, f, \Phi$ and $\Psi$ satisfying
the equality\[F=\Psi^{-1}\circ (f\times\id_{T})\circ\Phi.\]

As one may easily check, this makes the multiple
points $D^k(F)$ and $D^k(f\times \id_{T})$ isomorphic via
$\Phi^{-1}\times\stackrel{k}{\dots}\times\Phi^{-1}$. Moreover,
the multiple points in $D^k(f\times \id_{T})$ are of the form
$((x_1,t),\dots,(x_k,t))$. Forgetting all but one of the copies of $t$
gives an isomorphism \cite[Proposition 3.8]{Nuno-Ballesteros2015On-multiple-poi} \[\Omega\colon D^k(f\times\id_{T})\to D^k(f)\times
T.\] In turn, the spaces $\cY$ and $Y\times T$ are isomorphic via $\Psi$.

There are two geometric considerations where the $k$-fold product structure
of $\Phi^{-1}\times\stackrel{k}{\dots}\times\Phi^{-1}$ plays a role:
First of all, the isomorphisms and the $\varepsilon^k$ mappings are compatible
in the sense that there is a commutative diagram
 \[
\begin{tikzcd} \phantom{Aaa} D^k(F)\phantom{Aaa}
\arrow{r}{\Omega\circ(\Phi^{-1}\times\stackrel{k}{\dots}\times\Phi^{-1})}
\arrow{d}{\varepsilon^k} &\phantom{Aaa}D^k(f)\times
T\phantom{Aaa}\arrow{d}{\varepsilon^k\times \id_T}\\
\phantom{Aaa}\cY\phantom{Aaa}\arrow{r}{\Psi}& \phantom{Aaa}Y\times
T.\phantom{Aaa} \end{tikzcd} 
\] 
Second, the map 
$\Omega\circ(\Phi^{-1}\times\stackrel{k}{\dots}\times\Phi^{-1})$ is
compatible with 
the action of the symmetric group $S_k$ on both
$D^k(F)$ and on $D^k(f)\times T$ (the latter
induced by the action of $S_k$ on $D^k(f)$).
This makes the construction of the $\Alt$ operators on
both sides equivalent and the pushforward
$\Psi_*$ takes the sheaf $\Alt \varepsilon^k_*\C_{D^k(F)}$ on $\cY$
to the sheaf $\Alt (\varepsilon^k\times \id_T)_*\C_{D^k(f)\times T}$
on $Y\times T$.

We can now show that  $y$ is not in the support of $\phi_\pi \left( \Alt
\varepsilon^k_*\C_{D^k(F)}\right)$ by checking the vanishing of  the stalk
cohomology: Take a suitable ball $B_r$ in $V$ around
$y$. Then, for $\delta$ sufficiently small we find
 \begin{eqnarray*}
   & & \mathcal H^i\left(\phi_\pi \left(\Alt \varepsilon^k_*\C_{D^k(F)}\right) \right)_{y}\\
    &=& \mathbb H^{i+1}\left(\mathcal Y\cap B_r, \mathcal Y\cap \pi^{-1}(\delta)\cap B_r;
      \Alt \varepsilon^k_*\C_{D^k(F)}  \right)\\
    & \cong & \mathbb H^{i+1}\left((Y\times T)\cap B_r, (Y\times
    \{\delta\})\cap B_r ; \Alt \varepsilon^k_*\C_{D^k(f)\times T} \right)\\
    &=& 0.
  \end{eqnarray*}
  The vanishing in the last line follows from the fact that 
  due to the product structure clearly 
  $(Y \times T) \cap B_r$ retracts onto the fiber $(Y \times \{\delta\}) \cap B_r$ 
  for $r \gg \delta > 0$.
\end{proof}

Now we can determine where the non-trivial entries of the first 
page of the spectral sequence in Lemma \ref{lem:SpectralSequenceForAlternatingCohomology}
are concentrated.

\begin{prop}\label{propFirstPage}
Let $F$ be a dimensionally correct one-parameter unfolding of a germ $(\C^n,0)\to(\C^p,0)$ with $p> n$. 
  \label{prp:nonZeroEntriesFirstPage}
  For every integer $k\geq 2$, the nonzero cohomologies 
  \[
    \mathcal H^i\left(\phi_\pi \left( \Alt \varepsilon^k_*\C_{D^k(F)}\right) \right)_{0}
  \] 
  are concentrated in degrees $i\geq 0$ with 
      \[
      kn-(k-1)p-d \leq i \leq kn-(k-1)p.
    \]
   In the case $k=1$ the above vanishing cycles are all zero. 
\end{prop}
\begin{proof}
  Since $F$ is dimensionally correct, the sheaves
  \[
    \Alt \varepsilon^k_*\C_{D^k(F)} [kn-(k-1)p+1]
  \]
  are perverse, and, by virtue of 
  Theorem \ref{thm:vanishingCycleFunctorPreservesPerversity}, 
  so are the sheaves $$\phi_\pi\left(\Alt \varepsilon^k_*\C_{D^k(F)}\right)[kn-(k-1)p]$$ 
  on $Y_0$. Since these sheaves
  are supported on a space of dimension at most $d$, it follows 
  from Proposition \ref{prp:stalkCohomologyBoundFromPerversity} that
  the non trivial cohomologies of their stalks at the origin are 
  concentrated in degrees $i$ with $-d \leq i \leq 0$. 

  Shifting everything back by $kn-(k-1)p$, we obtain the desired bounds 
  for $k\geq 2$.
  Note that, since the alternating cohomology
  \begin{eqnarray*}
    \mathcal H^i\left( 
    \phi_{\pi}\left(  \Alt \varepsilon_*^k \C_{D^k(F)}  \right) \right)_0
    &=& \tilde{H}^i_{\Alt} (D^k(f_\delta))
    \subset \tilde{H}^i(D^k(f_\delta))
  \end{eqnarray*}
  is a subspace of the singular reduced cohomology of the disentanglement, 
  there can be no contributions in negative degrees. 

  For $k=1$ the disentanglement always provides a product structure 
  \[(D^1(F),(0,0)) \cong (X \times T,(0,0))\] and therefore 
  \[
    \mathcal H^i\left(\phi_\pi \left(\Alt \varepsilon^1_*\C_{D^1(F)}\right)\right)_{0}
    = H^i(X\times T,X\times \delta)=0.
  \]
  This finishes the proof.
\end{proof}

\begin{proof}[Proof of Theorem \ref{KM}]
The proof follows by the same
spectral sequence argument as in \cite[Theorem 1.1]{PZ}:
The concentration of cohomology follows immediately
from the isomorphism
\[
  \mathcal H^\ell(\phi_\pi\C_{\mathcal Y})_0\cong
  \bigoplus_i E^{i,\ell-i}_\infty
\]
 from Remark \ref{remIsomInftyPage} 
and the vanishing of the entries of the first page 
coming from Proposition \ref{prp:nonZeroEntriesFirstPage}.
\end{proof} 
Let us give an example which illustrates the general
situation:
\begin{ex}
  Let $f\colon(\C^{16},0)\to(\C^{21},0)$
  be a dimensionally correct map whose instability locus has
  dimension $d=2$.  According to Proposition \ref{propFirstPage}, the possibly nonzero entries in the first page of the spectral sequence $E_1^{i,j}$ from Lemma 
  \ref{lem:SpectralSequenceForAlternatingCohomology} are the following:
 \begin{equation*}
  \begin{array}{ccccccccccccccccccccccccccccc}
 \multicolumn{1}{c|}{i \uparrow}	&
       &    &    &     \\
 \multicolumn{1}{c|}{}	&
       &    &    &     \\
  \multicolumn{1}{c|}{11} &
      & H^{11}_{\Alt}(D^2(f_\delta)) &   &   & \\
  \multicolumn{1}{c|}{10} &
      \phantom{\vdots}    & H^{10}_{\Alt}(D^2(f_\delta)) &   &   \\
  \multicolumn{1}{c|}{	  9} &
   \phantom{\vdots}       & H^{9}_{\Alt}(D^2(f_\delta))&    &     \\
  \multicolumn{1}{c|}{	  8} &
     \phantom{\vdots}    &    &    &   
       \\
  \multicolumn{1}{c|}{	  7 }&
   \phantom{\vdots}      &    &    &   
       \\
   \multicolumn{1}{c|}{  6} &
  \phantom{\vdots}       &   & H^{6}_{\Alt}(D^3(f_\delta)) &    
       \\
  \multicolumn{1}{c|}{	  5 }&
 \phantom{\vdots}         &    & H^{5}_{\Alt}(D^3(f_\delta)) &    
       \\
  \multicolumn{1}{c|}{	  4} &
 \phantom{\vdots}        &   & H^{4}_{\Alt}(D^3(f_\delta)) &    
       \\
  \multicolumn{1}{c|}{	  3} &
 \phantom{\vdots}        &   &    &   
       \\
  \multicolumn{1}{c|}{	  2} &
 \phantom{\vdots}        &   &    & 
       \\
  \multicolumn{1}{c|}{	  1} &
 \phantom{\vdots}        &   &   & H^{1}_{\Alt}(D^4(f_\delta)) 
       \\
  \multicolumn{1}{c|}{	  0} &
       \phantom{H^{10}_{\Alt}(D^2(f_\delta))}  & \phantom{\vdots}  &   & H^{0}_{\Alt}(D^4(f_\delta))  \\
\hline	\multicolumn{1}{c|}{}
& 0 & 1 & 2 & 3 &\underset{j}{\longrightarrow}
  \end{array}
\end{equation*}

In this case, the sequence necessarily collapses on this first page. In
general, the differentials between the nonzero entries on the first
page can lead to cancellations for the following ones and on the limit
page. In either case, the positions of the nonzero entries on the first 
page give a bound on the nonzero entries of 
$\bigoplus_i E_\infty^{i,k-i} \cong \mathcal
H^k(\phi_\pi\C_{\mathcal Y})_0$. 
For this particular example we see that the non-trivial cohomologies
of a disentanglement $Y_\delta$ are concentrated on the degrees $\ell\in
\{3,4\}\cup\{6,7,8\}\cup\{10,11,12\}.$ 

\end{ex}

\section{Monodromy for disentanglements}\label{secMonodromy}

Our study of the monodromy for disentanglements in 
Theorem \ref{monodromy} will be based on the monodromy 
on the multiple point spaces and their alternating cohomology. This 
requires an adaptation of the definitions to fibrations with 
group actions. 

Note that any stable one parameter unfolding 
\[
  F \colon (\C^n,0) \times (\C,0) \to (\C^p,0)\times (\C,0)
\]
of a finite map 
germ $f$ with $p>n$ as in Theorem \ref{monodromy}
gives rise to analytic fibrations 
\begin{equation}
  \pi \circ \varepsilon^k \colon D^k(F) \cap (\varepsilon^k)^{-1}(B \times D\setminus\{0\}) 
  \to D\setminus\{0\}
  \label{eqn:MultiplePointsFibrations}
\end{equation}
for suitable choices of a representative $F$, a ball $B \subset \C^p$ in 
the target of $f$, and a disc $D \subset\C$ in 
the parameter space.
By construction, the symmetric group $\mathfrak S_k$ acts fiberwise on the multiple point 
space $D^k(F) \cap (\varepsilon^k)^{-1}(B\times D)$, for each $k>0$.

As already noted by K. Houston in \cite[Lemma 2.14]{Houston97},
there exist \textit{equivariant} 
local trivializations. 
We may therefore suppose that the trivialization of the pullback of the fibration 
(\ref{eqn:MultiplePointsFibrations}) to the universal cover 
$\exp \colon S \to D\setminus\{0\}$ is $\mathfrak S_k$-equivariant and 
that the parallel transport of the fiber 
\[
  h \colon D^k(f_\delta) \to D^k(f_\delta) 
\]
along a closed, counterclockwise loop in $D^*$ around the origin 
commutes with the $\mathfrak S_k$-action on $D^k(f_\delta)$. 
It is easy to see that this implies that also the maps induced by $h$ on cohomology 
\[
  h^i \colon H^i( D^k(f_\delta)) \to 
  H^i( D^k(f_\delta) )
\]
restrict to operators on the alternating cohomology 
$H^i_{\Alt}(D^k(f_\delta))$. We therefore have 
a well defined monodromy action  
\[
  h \colon 
  \phi_\pi\left(\Alt \varepsilon_*^k \C_{D^k(F)} \right)
  \to \phi_\pi\left(\Alt \varepsilon_*^k \C_{D^k(F)} \right)
\]
on the vanishing cycles of the alternating sheaves. As we shall see now, this gives the vanishing cycles of 
the alternating sheaves the structure of a $\C[s, s^{-1}]$-module where 
$s$ is a formal variable standing for the monodromy automorphism $h$.

To relate monodromy to our usual spectral sequence, it is useful to introduce a point of view about eigenvalues and Jordan blocks slightly different from that of the introduction: Consider the monodromy of a fibration $
  E\xrightarrow{\pi} D^*$ over the punctured disk
with fibre $F$, and assume that the fibration was obtained as a small representative of an analytic germ $(E,0)\to(D,0)$, as in \cite{Le78}. Consider the one-variable Laurent polynomial ring $R:=\C[s, s^{-1}].$ The monodromy action gives each cohomology $H^i(F)$ an $R$-module structure, by declaring that multiplication by $s$ acts on $H^i(F)$  by the  monodromy automorphism $h^i\colon H^i(F)\to H^i(F)$. It is well known that $F$ is a homotopy equivalent to a finite CW-complex, and that each monodromy automorphism $h^i$ is idempotent. As a consequence, each $H^i(F)$ is a finitely generated torsion $R$-module. Note that $R$ is a principal ideal domain so that according to the structure theorem, every finitely generated torsion $R$-module $A$ can be uniquely decomposed as \[ A\cong \oplus_{j=1}^m R/(s-\lambda_j)^{a_j},\]
 where $\lambda_j $ are non-zero complex numbers and $a_j$ are positive integers. It turns out that the set of eigenvalues of $h^i$ can be recovered as \[\supp(A):=\{ \lambda_j \}_{1\leq j \leq m},\] and the maximal size of the Jordan blocks  is \[J(A):= \max_{1 \leq j \leq m} \{a_j\}.\]

Now we list two simple facts related to finite generated torsion $R$-modules, where the proofs are left to the reader. 
\begin{lem}\label{cm1}
Let $A \overset{u}{\to} B \overset{v}{\to}  C$ be a complex of finitely generated torsion $R$-modules. Then $\Ker v/ \im u$ is also a finitely generated torsion $R$-module and we have that \begin{itemize}
\item[(a)] $\supp(\Ker v/ \im u) \subset \supp( B) $,
\item[(b)] $J(\Ker v/ \im u) \leq J(B) $.
\end{itemize}
\end{lem}
\begin{lem}\label{cm2}
Let $0  \to A \to  B \to   C\to  0$ be a short exact sequence of finitely generated torsion $R$-modules. Then  we have that \begin{itemize}
\item[(a)] $\supp(B) = \supp( A) \cup \supp(C) $,
\item[(b)] $\max\{ J(A), J(C)\} \leq J(B) \leq J(A)+J(C) $.
\end{itemize}
\end{lem}

\begin{proof}[Proof of Theorem \ref{monodromy}]
The explanation in the beginning of this section shows that the spectral sequence
\[ E_1^{i,j}= \mathcal H^i\left(
   \phi_\pi\left(\Alt \varepsilon^{}_* \C_{D^{j+1}(F)}\right)
   \right)_0 \Rightarrow \tilde{H}^{i+j}(Y_\delta).\]
can be studied in the category of $R$-modules, where every entry in it is a torsion $R$-module.

The assumption that $f$ has corank one implies that the one-parameter unfolding $F$ of $f$ also has corank one. Since $F$ is stable by assumption, the spaces  $D^{k}(F)$ are smooth and have the right dimension. The generic fibre $D^k(f_\delta)$ coincides with the Milnor fiber of the germ of maps:
 \[D^{k+1}(F) \to T.\] 
Then the classical monodromy theorem about eigenvalues and size of Jordan blocks apply to the monodromy automorphisms $ 
H^i( D^{k} (f_\delta))\to H^i( D^{k} (f_\delta))$. Hence, by Lemma \ref{cm2}, the size of the Jordan blocks of the monodromy automorphism $H^i_{\Alt}( D^{k} (f_\delta))\to H^i_{\Alt}( D^{k} (f_\delta))$ is at most $i+1$, i.e., 
$$J(E^{i,k-1}_1)\leq i+1$$

Since the page $E^{\bullet, \bullet}_{r+1}$ is computed from $E^{\bullet, \bullet}_{r}$ by taking cohomologies, we conclude from Lemma \ref{cm1} that the infinity page satisfies the above  property as well. 
Note that $E_1^{i,k-1}=0$ when $k\leq 1$.
By the isomorphism  $H^\ell (Y_\delta)\cong \bigoplus_i E^{i,\ell-i}_\infty$ and Lemma \ref{cm2}, we obtain that 
\[J(H^\ell (Y_\delta))\leq \sum_{0\leq i <\ell} J(E^{i,\ell-i}_\infty) \leq \sum_{0\leq i < \ell} (i+1)= \ell(\ell+1)/2 .\]

When $f$ has only isolated instability locus and $p>n+1$, then Proposition \ref{propFirstPage} gives us that $E_1^{i,k-1}$ are concentrated in degrees $ i=kn-(k-1)p$. In particular, $p>n+1$ implies that the spectral sequence degenerates at $E_1$-page. Hence the 
 possibly non-trivial reduced cohomologies  are 
\begin{center}
 $\tilde{H}^{kn-(k-1)(p-1)}(Y_\delta)\cong \tilde{H}^{kn-(k-1)p}_{\Alt}(D^k(f_\delta))$ for $1< k \leq  \left\lfloor \frac{p}{p-n} \right\rfloor$
\end{center} 
 with the corresponding size of Jordan block at most $kn-(k-1)p+1$. 
\end{proof}

\begin{rem}
  \label{rem:RootsOfUnity}
  Observe that the proof of Theorem \ref{monodromy} can be used to show
  that the eigenvalues of $h^i\colon H^i (Y_\delta)\to H^i (Y_\delta)$ are
  roots of unity, directly from the classical result about eigenvalues
  of the monodromy of hypersurfaces (that is, with no resort  to the more
  general result of L\^e for fibrations on analytic spaces). Moreover,
  it shows that the eigenvalues of the monodromy on $H^i (Y_\delta)$ are
  a subset of the union of eigenvalues of the monodromies on $H^i_{\Alt}
  (D^k(f_\delta)),k\geq 0$.
\end{rem}

\begin{rem} We saw earlier in Remark \ref{remTwoPointSuspension} 
  how the study of Milnor fibers of hypersurfaces 
  $g \colon (\C^{n+1},0) \to (\C,0)$ was equivalent to the 
  study of disentanglements of bi-germs $f \colon (\C^n,\{p,q\})\to (\C^{n+1},0)$. 
  The above considerations show that this also comprises the monodromy:
  Since all the ingredients involved in the spectral sequence admit an $R$-module structure, 
  the isomorphism $\tilde H^i(M)\cong \tilde H^{i+1}(Y_\delta)$ is an isomorphism of 
  $R$-modules and therefore the monodromies on both sides are compatible 
  by the considerations in Remark \ref{remCorrespondenceMilnorDisentanglement}.
\end{rem}

\end{document}